\documentclass[a4paper,12pt,leqno]{amsart}
\usepackage[colorlinks=true,linkcolor=darkblue,citecolor=darkblue]{hyperref}

\usepackage{graphicx}
\usepackage{pinlabel} %for algebraic and geometric topology
\usepackage{amsmath}
\usepackage{amsfonts}
\usepackage{amssymb}
\usepackage{mathtools}
\usepackage[all,cmtip]{xy}
\usepackage{amsthm}
\usepackage{tikz-cd}
\usepackage{comment}
\usepackage{enumerate}
\usepackage{url}
\usepackage{epsfig}
\usepackage[utf8]{inputenc}
\usepackage{aliascnt}
\usepackage[capitalise]{cleveref}
\usepackage{geometry}
\usepackage{multicol}

\usepackage{tabularx} 
\usepackage{cancel}

\makeatletter
\providecommand\@dotsep{5}
\def\listtodoname{List of Todos}
\def\listoftodos{\@starttoc{tdo}\listtodoname}
\makeatother

\newtheorem{theorem}{Theorem}[section]
\newaliascnt{proposition}{theorem}
\newtheorem{proposition}[proposition]{Proposition}
\aliascntresetthe{proposition}
\newaliascnt{corollary}{theorem}
\newtheorem{corollary}[corollary]{Corollary}
\aliascntresetthe{corollary}
\newaliascnt{lemma}{theorem}
\newtheorem{lemma}[lemma]{Lemma}
\aliascntresetthe{lemma}
\newaliascnt{claim}{theorem}

\aliascntresetthe{claim}
\newaliascnt{fact}{theorem}

\aliascntresetthe{fact}

\newcommand{\mycomment}[1]{}

\newaliascnt{conjecture}{theorem}

\aliascntresetthe{conjecture}
  \theoremstyle{definition}
\newaliascnt{definition}{theorem}

\aliascntresetthe{definition}
\newaliascnt{example}{theorem}

\aliascntresetthe{example}
\newtheorem{remark}{Remark}
\newaliascnt{question}{theorem}

\aliascntresetthe{question}
\newaliascnt{notation}{theorem}
\newtheorem{notation}[notation]{Notation}
\aliascntresetthe{notation}

\crefname{theorem}{theorem}{theorems}
\Crefname{theorem}{Theorem}{Theorems}
\crefname{proposition}{proposition}{propositions}
\Crefname{proposition}{Proposition}{Propositions}
\crefname{corollary}{corollary}{corollaries}
\Crefname{corollary}{Corollary}{Corollaries}
\crefname{lemma}{lemma}{lemmas}
\Crefname{lemma}{Lemma}{Lemmas}
\crefname{claim}{claim}{claims}
\Crefname{claim}{Claim}{Claims}
\crefname{fact}{fact}{facts}
\Crefname{fact}{Fact}{Facts}
\crefname{conjecture}{conjecture}{conjectures}
\Crefname{conjecture}{Conjecture}{Conjectures}
\crefname{definition}{definition}{definitions}
\Crefname{definition}{Definition}{Definitions}
\crefname{example}{example}{examples}
\Crefname{example}{Example}{Examples}
\crefname{remark}{remark}{remarks}
\Crefname{remark}{Remark}{Remarks}
\crefname{question}{question}{questions}
\Crefname{question}{Question}{Questions}
\crefname{notation}{notation}{notations}
\Crefname{notation}{Notation}{Notations}

\newcommand{\calN}{{\mathcal N}}

\newcommand{\nbeq}{\begin{equation}}
\newcommand{\neeq}{\end{equation}}
\newcommand{\beq}{\begin{equation*}}
\newcommand{\eeq}{\end{equation*}}

\DeclareMathOperator{\cd}{cd}

\DeclareMathOperator{\vcd}{vcd}
\DeclareMathOperator{\cdfin}{\underline{cd}}

\DeclareMathOperator{\gd}{gd}
\DeclareMathOperator{\gdfin}{\underline{gd}}

\newcommand{\modnn}{\mathcal{N}_{g,n}}

\DeclareMathOperator{\Mod}{Mod}

\definecolor{darkblue}{rgb}{0.0, 0.0, 0.55}
\usepackage{verbatim}

\usepackage{listings}

\begin{document}

\title[]{The proper geometric dimension of the mapping class groups of non-orientable surfaces with punctures}

\author[N. Colin]{Nestor Colin}
\author[R. Jiménez Rolland]{Rita Jiménez Rolland}
\address{Instituto de Matemáticas, Universidad Nacional Autónoma de México. Oaxaca de Juárez, Oaxaca, México 68000}
\email{rita@im.unam.mx}
\email{ncolin@im.unam.mx}

\author[P. L. León Álvarez]{Porfirio L. León Álvarez}
\address{Centro de Investigación en Matemáticas, A.C., Unidad Mérida. Parque Científico y Tecnológico de Yucatán, Carretera Sierra Papacal--Chuburná Puerto Km 5.5, C.P. 97302, Sierra Papacal, Mérida, Yucatán, México}
\email{porfirio.leon@cimat.mx}

\author[L. J. Sánchez Saldaña]{Luis Jorge S\'anchez Salda\~na}
\address{Departamento de Matemáticas, Facultad de Ciencias, Universidad Nacional Autónoma de México}
\email{luisjorge@ciencias.unam.mx}

\date{\today}


\begin{abstract}
We study the proper geometric dimension of {\it full} mapping class groups of
non-orientable surfaces with punctures. Building on the computation for closed
non-orientable surfaces, we prove that the proper geometric dimension agrees
with the virtual cohomological dimension in all but a finite collection of
low-complexity cases, for which we obtain explicit bounds. Our results provide
the non-orientable counterpart of the corresponding computations for mapping
class groups of closed and punctured orientable surfaces.
\end{abstract}
\maketitle

%\tableofcontents

\section{Introduction}

For a discrete group $G$, a model for a classifying space of $G$ for proper actions $\underline{E}G$ is a contractible $G$-CW-complex  on which $G$ acts properly, and such that the fixed point set of a subgroup $H < G$ is contractible if $H$ is finite, and is empty otherwise. The {\it proper geometric dimension} of $G$, denoted by $\underline{\gd}(G)$,  is the minimal dimension of such models.

The proper geometric dimension of mapping class groups of closed orientable surfaces was computed by Aramayona and Mart\'{\i}nez-P\'{e}rez
\cite{AMP14}. The case of orientable surfaces with punctures was subsequently treated by the present authors in \cite{4Amigos_FullMCG}. For closed non-orientable surfaces, the corresponding computation was carried out by Hidber, S\'{a}nchez Salda\~na, and Trujillo-Negrete \cite{HSSTN24}.  The purpose of this paper is to study the remaining case of non-orientable surfaces with punctures.

Let \( N_g \) be a  non-orientable closed connected surface of genus \( g\geqslant 1 \) and  consider \( \{x_1, \ldots, x_n\} \)  a collection of \( n\geqslant 1 \) distinguished points in \( N_g\).   The {\it (full)  mapping class group} $\calN_{g,n}$ is the group of isotopy classes of diffeomorphisms  $f:N_{g,n}\rightarrow N_{g,n}$, where $N_{g,n}:=N_g-\{x_1, \ldots, x_n\}$.  Since $\calN_{g,n}$ is virtually torsion-free, its virtual cohomological dimension $\vcd(\calN_{g,n})$ is a lower bound for $\gdfin(\calN_{g,n})$. Previous work of the authors  implies equality, and the existence of a model for the classifying space for proper actions of minimal dimension, in the case when $n=1$ \cite[Corollary 1.2]{4Amigos_Spine_NO}, and for {\it pure}  mapping class groups \cite[Corollary 1.3]{4Amigos_Spine_NO}. Our main result is the following:

\begin{theorem}\label{Thm:Main} The proper geometric dimension of $\calN_{g,n}$ is 
\[ \underline{\gd}(\calN_{g,n})= \vcd(\calN_{g,n}), \]
 when $g=1$ and either $1\leqslant n\leqslant 3$ or $n\geqslant 6$;  $g=2$ and either $n=1$ or $n\geqslant 3$, or $g\geqslant 3$ and $n\geqslant 1$.  Moreover, there exists a  cocompact model for $\underline{E}\calN_{g,n}$ of dimension equal to $\vcd(\calN_{g,n})$. 
 
Furthermore, we obtain the following estimates for the remaining low-genus cases:\begin{enumerate}
    \item for $g=1$ and $n=4,5$ we have
\[
\vcd(\calN_{1,n})\leqslant \underline{\gd}(\calN_{1,n})\leqslant 5,\text{ and}
\]
    \item for $g=2$ and $n=2$ we have
\[
\vcd(\calN_{2,2})\leqslant \underline{\gd}(\calN_{2,2})\leqslant \vcd(\calN_{2,2})+1.
\]
   
\end{enumerate}
\end{theorem}

The computation and estimates of  the proper geometric dimension follow from  \Cref{Thm: MainProjectivePlane} for $g=1$, \Cref{Thm: MainKB} for $g=2$,  \Cref{prop:N3:marked:points} for $g=3$, \Cref{Thm:g:4:5:punctures} for $g=4,5$, and \Cref{prop:Genus6} for $g\geqslant 6$. These theorems are the main results of \Cref{sec:GenusOne,sec:GenusTwo,sec:genus-3,sec:GENUS4&5,sec:GENUS6} respectively, which gives an overview of the structure of the paper. 

 On the other hand,  we show the existence of a cocompact model for $\underline{E}\calN_{g,n}$ in \Cref{prop:nonorientable-spine-cocompact} below. Hence, the {\it moreover} part of the main result follows from our computation of $\gdfin(\calN_{g,n})$ and  a result of Lück that we recall in \Cref{thm:Luck}.

\subsection{Proper geometric dimension of braid groups on the projective plane} The {\it (full) $n$-th strand braid group $B_n(N_1)$ on the projective plane} is the fundamental group of the $n$th unordered configuration space of $N_1=\mathbb{RP}^2$.  For $n\geqslant 1$ these groups are virtually torsion free and have $\vcd(B_n(N_1))=\max\{0,n-2\}$ \cite[Theorem 5]{VcdBraid}. Moreover, the mapping class group $\calN_{1,n}$ is related to the  $B_n(N_1)$ by the following central extension (see for example \cite[Section 2.4]{BraidSurvey})
\[ 1\rightarrow C_2\rightarrow B_n(N_1)\rightarrow \calN_{1,n}\rightarrow 1,\]
when $n\geqslant 2$.  
It follows from \cite[Lemma 5.8]{Lu05} that $\gdfin(B_n(N_1))=\gdfin(\calN_{1,n})$, and \Cref{Thm:Main} implies:

\begin{corollary} For $1\leqslant n \leqslant 3$ or $n\geqslant 6$, 
$\gdfin(B_n(N_1))=\vcd(B_n(N_1))=\max\{0,n-2\}$. Moreover, there exists a cocompact model for $\underline E B_n(N_1)$ of dimension  $\vcd(B_n(N_1))$.
\end{corollary}

\subsection{Overview of the proof and  the paper} To prove our main result \Cref{Thm:Main}, we obtain  $\gdfin(\calN_{g,n})$ by computing its algebraic counterpart the {\it proper cohomological  dimension} $\cdfin(\calN_{g,n})$. We do it by following the strategy of \cite{AMP14}; see also \cite{4Amigos_FullMCG}. This computation amounts to verifying that  the inequality $\vcd(WF) +\lambda (F)\leqslant \vcd(\calN_{g,n})$ holds  for any finite subgroup $F$ of $\calN_{g,n}$, and using \cite[Theorem 3.3]{AMP14} (see \Cref{Conchita:criterion} below). Here $WF$ denotes the {\it Weyl group} of $F$ in $\calN_{g,n}$  and $\lambda (F)$ is the {\it length} of $F$.  We recall the necessary definitions and results in \cref{sec:dimensions}.

In \Cref{sec:vcdBOUND}, we use Nielsen realization to obtain upper bounds for $\vcd(WF)$ in terms of the virtual cohomological dimension of the mapping class group of a suitable surface whose topological invariants come from the orbifold quotient $N_{g,n}/F$.  Most of the present paper deals with either computing or upper-bounding these invariants. On the other hand, $\lambda(F)$ can be bounded by the number of factors in the prime decomposition of $|F|$.

An independent and careful analysis of the finite subgroups of $\calN_{g,n}$ is required for  $1\leqslant g\leqslant 3$. This is done in \Cref{sec:GenusOne,sec:GenusTwo,sec:genus-3}, respectively. For genus $g=1,2$ and small $n$ we can only obtain upper bounds for $\underline{\gd}(\calN_{g,n})$ from our method; see for instance \Cref{rmk:FailsGenus1,rmk:FailsGenus2}. In the case of genus $g=3$, \cite[Proposition 1]{BEM14} gives the list of finite groups of diffeomorphisms acting  on a non-orientable surface $N_3$. From this list we obtain in \Cref{table:orbifolds-genus-3}, as part of our argument, the possible signatures of quotient orbifolds of finite group actions on $N_3$ that satisfy the Riemann--Hurwitz equation. In \Cref{sec:GENUS4&5},  we prove \Cref{Thm:g:4:5:punctures} for genus $g=4,5$ as a consequence of results from \cite{HSSTN24}.  For $g\geqslant 6$,  we follow the same strategy from \cite[Section 6]{4Amigos_FullMCG} in \Cref{sec:GENUS6} and obtain \Cref{prop:Genus6} as a straightforward application of results in \cite{HSSTN24} for $\calN_g$. 

 Finally, in Appendix \ref{Appendix:cocompactness-spines} we prove the cocompactness of {\it Harer's spine} associated to punctured orientable and non-orientable surfaces. Appendices \ref{Appendix:code:projective-plane}, \ref{Appendix:code:n:F} and \ref{Appendix:code:genus-3} include the programs used in Sections \ref{sec:GenusOne},\ref{sec:GenusTwo} and \ref{sec:genus-3}, respectively.\bigskip

\noindent{\bf IA use declaration:} The ideas, organization and  writing of the paper are due to the authors. ChatGPT was used for proofreading and language editing; to generate an initial version of the code presented in Appendices \ref{Appendix:code:projective-plane}, \ref{Appendix:code:n:F} and  \ref{Appendix:code:genus-3},  and to  slightly improve the upper bounds in \Cref{Lem:Adding:Marked:Points} and \Cref{Prop:Bound:WF}. These were independently verified and modified by the authors to obtain the current text in the paper. \medskip

\noindent{\bf Acknowledgments.}  The first author was funded by SECIHTI through the program \textit{Estancias Posdoctorales por México.} The third author's work was supported by UNAM \textit{Posdoctoral Program (POSDOC)}. All authors are grateful for the financial support of DGAPA-UNAM grant PAPIIT IN102426. The collaboration among the four authors began at the 2023 edition of the Pan Pacific International Conference on Topology and its Applications (PPICTA), held in Puebla, Mexico. At the 5th PPICTA held in Osaka, Japan in 2026, the fourth author gave a plenary lecture in which the results of this article were announced. We thank the organizers of the conference for providing opportunities to present our work and foster collaborations.

\section{Preliminaries}

\subsection{Proper geometric and cohomological dimensions}\label{sec:dimensions}
We recall here some notions of dimension defined for a given group $G$ that will be used in this paper.

A model for the {\it classifying space of $G$ for proper actions $\underbar{E}G$} is a $G$-CW-complex $X$  such that the fixed point set $X^H$ of a subgroup $H < G$ is contractible if $H$ is finite, and is empty otherwise. Such a model always exists and is unique up to proper $G$-homotopy. The {\it proper geometric dimension of $G$}, denoted
$\underline{\gd}(G)$, is the smallest integer $n$ for which there exists an $n$-dimensional model for $\underline EG$.

The proper cohomological dimension of $G$ is the Bredon cohomological dimension with
respect to the family of finite subgroups. More precisely, let {\sc Fin} be the collection of finite subgroups of $G$, and let
$\mathcal O_{\text{\sc Fin}}G$ be the orbit category whose objects are the
homogeneous $G$-spaces $G/H$ with $H\in\text{\sc Fin}$, and whose morphisms are
$G$-maps. A $\mathcal{O}_{\text{\sc Fin}}G$-module is a contravariant functor from $\mathcal{O}_{\text{\sc Fin}}G$ to the category of abelian groups. The
\emph{proper cohomological dimension} of $G$, denoted $\underline{\cd}(G)$, is
the projective dimension of the constant $\mathcal{O}_{\text{\sc Fin}}G$-module
$\mathbb Z_{\text{\sc Fin}}$. Equivalently, it is the largest integer $n$ for which
there is a $\mathcal{O}_{\text{\sc Fin}}G$-module $M$ with
$H^n_{\text{\sc Fin}}(G;M)\neq 0$.

The {\it cohomological dimension $\cd(H)$} of a group $H$ is the length of the shortest projective resolution, in the category of $H$-modules, for the trivial $H$-module $\mathbb{Z}$. If  $G$ is virtually torsion free, then $G$ contains a torsion-free subgroup $H$ of finite index, and the {\it virtual cohomological dimension of $G$} is defined to be $\vcd(G) = \cd(H)$.  This is well defined and it does not depend on the choice of the finite index torsion free subgroup $H$ of $G$ by a well-known theorem of Serre. For every virtually torsion free group $G$, we have the following inequalities 
\begin{equation}\label{eq:InDim}
    \vcd(G)\leqslant \cdfin(G) \leqslant \gdfin(G) \leqslant \max\{3, \cdfin(G)\},
\end{equation}
see for instance \cite[Theorem 2]{BLN01}. In particular, if $\underline{\cd}(G)\geqslant 3$, then $\underline{\gd}(G)=\underline{\cd}(G)$, and the only possible Eilenberg--Ganea phenomenon for proper actions occurs
when $\underline{\cd}(G)=2$ and $\underline{\gd}(G)=3$.

The following theorem is a consequence of \cite[Theorem 13.19]{Lu89}, as explained in \cite[Section 3]{AMP14}.
\begin{theorem}\label{thm:Luck} Let $G$ be  a group with $\cdfin(G)=d\geqslant 3$. Then there exists a $d$-dimensional $\underbar{E}G$. If $G$ has a cocompact model for $\underbar{E}G$, then it also admits a cocompact $\underbar{E}G$ of dimension $d$.
\end{theorem}

%\vskip 10pt
 
Now, for $F$ a finite subgroup of $G$ we denote by $NF$ and $WF=NF/F$  the normalizer and the {\it Weyl group of} $F$, respectively. On the other hand, the \emph{length} $\lambda(F)$ of a finite group $F$ is the largest $i\geqslant 0$ for which  there exists a chain of subgroups \[1=F_0<F_1 < \cdots F_{i-1}< F_i=F.\]

For an integer $k\geqslant 2$, let  $\Omega(k)$ denote the number of prime factors of $k$, counted with multiplicity. The following lemma is a consequence of Lagrange's theorem and it will be useful in our computations below; see \cite[Lemma~3.1]{HSSTN24}.

\begin{lemma}\label{lem:UpperBoundLength}
    Let $G$ be a group and let $F$ be a finite subgroup. Then
    \[\lambda(F)\leqslant \Omega(|F|)\leqslant \log_2(|F|).\]
\end{lemma}

Given \Cref{eq:InDim}, to obtain the proper geometric dimension $\gdfin(G)$ of the group of interest we will compute its algebraic counterpart $\cdfin(G)$  following the general strategy of \cite{AMP14}; see also  \cite{HSSTN24} and \cite{4Amigos_FullMCG}. The next key result is a  mild generalization of \cite[Theorem~3.3]{AMP14}.

\begin{theorem}\label{Conchita:criterion}
 Let $G$ be a virtually torsion-free group. Suppose that there exists $m\geqslant 0$ such that for any finite subgroup $F<G$ we have \[\vcd(WF)+ \lambda(F)\leqslant m\] then $\cdfin(G)\leqslant m$. In particular, if $m=\vcd(G)$ then $\cdfin(G)=\vcd(G)$.   
\end{theorem}

\subsection{Mapping class groups}\label{Sec:MCG}
Let \( \Sigma_{g}^b \) be a possibly non-orientable  connected surface of genus \( g \) with $b\geqslant 0$ boundary components.  For $n\geqslant 0$, take $X=\{x_1,x_2,\ldots,x_n\}$ a set of {\it marked points}, which may be empty, in the interior of the surface. The {\it (full) mapping class group} $\Mod({\Sigma_{g}^b};X)$  is the group of isotopy classes of self-diffeomorphisms (orientation-preserving if the surface is orientable) of $\Sigma_{g}^{b}$ which take the set of distinguished points $X$ to itself and fix the boundary components pointwise. It is isomorphic to the mapping class group $\Mod({\Sigma_{g,n}^b})$ of the {\it punctured} surface $\Sigma_{g,n}^b:=\Sigma_{g}^b-X$. The {\it pure mapping class group}  is the subgroup of mapping classes that fix each marked point in $X$. 
\begin{notation}
We write $S_{g}^b$  if the surface is orientable and $N_{g}^b$ if it is non-orientable, and we will frequently use the more compact notation for mapping class groups
\[\Gamma_{g,n}^{b}:=\Mod(S_{g,n}^{b})\text{\ \ \ \ \  and \ \ \ \ \ }\calN_{g,n}^{b}:=\Mod(N_{g,n}^{b}).\]
We also denote  by $P\Gamma_{g,n}^{b}$ and $P\calN_{g,n}^{b}$  the corresponding  pure mapping class groups. Whenever there are no marked points or no boundary components, we omit
the corresponding index from the notation; for instance,
$\calN_{g,n}=\calN_{g,n}^{0}$ and $\calN_g=\calN_{g,0}^{0}$.\end{notation}

The group $\Mod(\Sigma_{g,n}^b)$ is virtually torsion free, and its virtual cohomological dimension  was computed by Harer for orientable surfaces and by Ivanov for non-orientable surfaces; see \cite[Proposition~2.4]{HSSTN24} and references therein.

\
\[
\vcd(\Gamma_{g,n}^{b})=
\begin{cases}
\max\{0,b-1\}, & \text{if } g=0 \text{ and } n+b\leqslant 2,\\
n+2b-3, & \text{if } g=0 \text{ and } n+b\geqslant 3,\\
1, & \text{if } g=1 \text{ and } n=b=0,\\
n+2b, & \text{if } g=1 \text{ and } n+b\geqslant 1,\\
4g-5, & \text{if } g\geqslant 2 \text{ and } n=b=0,\\
4g+n+2b-4, & \text{if } g\geqslant 2 \text{ and } n+b\geqslant 1,
\end{cases}
\]\smallskip

\[
\vcd(\calN_{g,n}^{b})=
\begin{cases}
\max\{0,b-1\}, & \text{if } g=1 \text{ and } n+b\leqslant 1,\\
n+2b-2, & \text{if } g=1 \text{ and } n+b\geqslant 2,\\
n+2b, & \text{if } g=2,\\
2g-5, & \text{if } g\geqslant 3 \text{ and } n=b=0,\\
2g+n+2b-4, & \text{if } g\geqslant 3 \text{ and } n+b\geqslant 1.
\end{cases}
\]
By a direct inspection of the $\vcd$ formulas above we obtain the following inequalities.

\begin{lemma}\label{Lem:Adding:Marked:Points}
Let $g,r,s,b,t\geqslant 0$ be integers. The following inequalities hold:
\begin{enumerate}
    \item 
$\vcd(\Mod(\Sigma_{g,r+s}^{b}))\leqslant \vcd(\Mod(\Sigma_{g,r}^{b}))+s$\   for $r+b>0$ and $g\geqslant 0$,
 \item  $\vcd(\Mod(\Sigma_{g,s}))\leqslant \vcd(\Mod(\Sigma_{g}))+s+1$ \   for  $g\geqslant 0$,
\item $\vcd(\Gamma_{g,r+s}^{b})\leqslant \vcd(\Gamma_{g,r}^{b})+s$\   
    for $g=0,1$ and $r,b\geqslant 0$, 
\item  $\vcd(\calN_{g,r+s}^{b})\leqslant \vcd(\calN_{g,r}^{b})+s$ \    for $g=1,2$ and $r,b\geqslant 0$,
\item$\vcd(\Mod(\Sigma_{g,r-t}^{b+t})) \leqslant\vcd(\Mod(\Sigma_{g,r}^{b}))+t$ \ if $t\leqslant r$.
\end{enumerate}
\end{lemma}
\begin{proof}
If $r+b>0$, then passing from $r$ to $r+s$ distinguished points either increases
the $\vcd$ exactly by $s$, or  one of the
following exceptional cases, where the $\vcd$ does not change, occurs: 
$\Gamma_{0,r}^{b}\text{ with }r+b\leqslant 2$, or $\calN_{1,r}^{b}\text{ with }r+b\leqslant 1$. Hence, the first inequality holds.

If $r=b=0$, the inequality can be checked directly from the formulas above. Notice that, in general, the formulas change  when passing from a closed surface to a punctured surface. That happens, for instance, in the cases  $\Gamma_{g}$ with $g\geqslant 1$ and $\calN_{g}$ with $g\geqslant 3$.

 For
$\Gamma_{0,r}^{b}$, the $\vcd$ is
$\max\{0,b-1\}$ when $r+b\leqslant 2$ and $r+2b-3$ when
$r+b\geqslant 3$. Then adding $s$ marked points increases the $\vcd$ by
at most $s$, including the case $r=b=0$. For $\Gamma_{1,r}^{b}$, the $\vcd$ is $1$ when
$r=b=0$ and $r+2b$ otherwise. Adding $s$ marked points increases the
dimension by at most $s$.

For $\calN_{1,r}^{b}$, the $\vcd$ is
$\max\{0,b-1\}$ when $r+b\leqslant 1$ and $r+2b-2$ when
$r+b\geqslant 2$. Then adding $s$ marked points increases the dimension by at
most $s$. Finally,  $\vcd(\calN_{2,r}^{b})=r+2b$, so adding $s$
marked points increases the dimension exactly by $s$.

For (5), the case $t=0$ is immediate. Suppose that $t>0$. Capping $t$ boundary components of $\Sigma_{g,r-t}^{b+t}$ with a punctured disc gives a short exact sequence
\[
    1\rightarrow\mathbb Z^k
\rightarrow\Mod(\Sigma_{g,r-t}^{b+t})
\rightarrow H \rightarrow1.
\]
where $k\leq t$ and $H\leqslant \Mod(\Sigma_{g,r}^{b} )$. Thus 
\[
    \vcd(\Mod(\Sigma_{g,r-t}^{b+t}))
\leqslant k+\vcd(H)
\leqslant t+\vcd(\Mod(\Sigma_{g,r}^{b})).
\] 
\end{proof}

\subsection{Quotient orbifold  and its mapping class group} Let $F$ be a finite subgroup of the mapping class group $\modnn$. By the Nielsen realization problem for non-orientable surfaces \cite[Theorem~1]{CX24}, there exists a dianalytic structure on $N_g$ such that $F$ is isomorphic to a finite group of automorphisms of $N_g$ with respect to this structure, and such that the automorphisms fix the set of marked points. Denote by $O_F$ the {\it quotient orbifold} $N_g/F$ with {\it underlying surface} $\Sigma_F$. Then there is a regular branched cover
$ p \colon N_g \to O_F $
between orbifolds. 

Denote by $X=\{x_1,\ldots,x_n\}$ the set of marked points on $N_g$, and set $Y := p(X)$. We will use the following notation:
\begin{itemize}
    \item $B$ denotes the set of branched points of the branched cover $p\colon N_g\to O_F$.
    \item $g_F$ denotes the genus of $\Sigma_F$.
    \item $e_F$ denotes the number of elliptic points of $O_F$.
    \item $c_F$ denotes the number of corner points of $O_F$.
    \item $b_c$ denotes the number of boundary components of $\Sigma_F$ that contain at least one corner point.
  \item $b_m$ denotes the number of boundary components of $\Sigma_F$ that do not contain any corner points; equivalently, all points on such boundary components are mirror points.
    \item $b=b_m+b_c$ denotes the total number of boundary components of $\Sigma_F$.
    \item $n_F$ denotes the number of points of $Y$ that are neither elliptic points, corner points, nor mirror points.
    \item $b_F$ denotes the number of boundary components that contains at least one point of $Y$ and no corner points.
\end{itemize}

Moreover, the {\it signature of $O_F$} is given by  \[\big(g_F;\pm;[m_1,m_2,\ldots,m_{e_F}]; \{ (n_{i,1},\ldots, n_{i,s_i}) , i=1,\ldots,b\}\big),\]
where we write $+$ if $\Sigma_F$ is orientable and $-$ if it is non-orientable,  the $m_i$'s denote the order of the elliptic points of $O_F$ for $i=1,2,\ldots,{e_F}$, and $n_{i,j}$ is the ramification index of the $j$th corner point in the $i$the boundary component for $i=1,\ldots,b$. 
We use the notation $(-)$ when the corresponding boundary component does not have corner points. If $O_F$ has no elliptic points, we write $[-]$, and if it has no boundary components, i.e. $b=0$, we write $\{-\}$.

The {\it mapping class group of the orbifold $O_F$}, that we denote by $\Gamma_F^*$ following \cite{HSSTN24},  is the group of isotopy classes of self-diffeomorphisms of $O_F$ that preserve the orbifold structure. More precisely, elliptic points of a given order  are mapped to elliptic points of the same order, mirror points to mirror points, and corner points to corner points of the same order. If, additionally, we require the self-diffeomorphism to preserve the set of marked points $Y$ we denote the corresponding group of isotopy classes by $\Gamma_F^*(Y)$. Observe that marked points may be non-singular (free) points, elliptic points, mirror points, or corner reflectors, and in all cases they are required to be sent to marked points of the same type.

In the definition of the mapping class group of a surface with
punctures, we can think of the punctures as boundaries that are not fixed pointwise by diffeomorphisms and where isotopies can move the points of these boundaries. On the other hand, whenever a diffeomorphism fixes pointwise a non-empty set of points in a boundary component, up to isotopy and up to a finite power, we can think that such a diffeomorphism is fixing  pointwise the boundary component. With this in
mind, we think of the underlying topological surface $\Sigma_F$ of the orbifold $O_F$ as a surface of genus $g_F$ with $b_c+b_F$ boundary components (the boundaries of $\Sigma_F$ that have at least one corner point or a marked point of $Y$) and $e_F + (b_m-b_F) +n_F$ punctures (or marked points).  Notice that a finite index subgroup of the orbifold mapping class group $\Gamma_F^*(Y)$ is a subgroup of 
\begin{equation}\label{eq:ModUnderlying}
\Mod\big(\Sigma_{g_F,b_m+e_F+n_F-b_F}^{b_c+b_F}\big)=
\begin{cases}
\calN_{g_F,b_m+e_F+n_F-b_F}^{b_c+b_F}, & \text{if } \Sigma_F \text{ is non-orientable},\\ 
\\
\Gamma_{g_F,b_m+e_F+n_F-b_F}^{b_c+b_F}, & \text{if } \Sigma_F \text{ is orientable}.
\end{cases}
\end{equation}

\subsection{Weyl groups of finite subgroups of mapping class groups}\label{sec:vcdBOUND} 
In this section we obtain certain upper bounds for the $\vcd$ of Weyl groups of finite subgroups of the mapping class groups $\modnn$ that will be useful to prove our main result. We use the notation from the previous sections.

From \cite[Proposition 2.3 $\&$ Section 2.1]{MAHER} it follows that for any finite subgroup $F$ of $\Gamma_{g,n}$, its Weyl group $WF$ is commensurable with $\Gamma_{g_F,e_F+n_F}$. In the non-orientable case, the following inequality is enough for our computations below; see also \cite[Theorem 2.6]{HSSTN24}.

\begin{lemma}\label{lem:vcd:Weyl:Group}
Let $g+n-2>0$ and consider a finite subgroup $F$ of $\modnn$. The Weyl group $WF$ is isomorphic to a subgroup of $\Gamma_F^*(Y)$, and its virtual cohomological dimension is bounded above by
\[
\vcd(WF)
\leqslant
\begin{cases}
\vcd\bigl(\calN_{g_F,b_m+e_F+n_F-b_F}^{b_c+b_F}\bigr), & \text{if } \Sigma_F \text{ is non-orientable}, \\
\\
\vcd\bigl(\Gamma_{g_F,b_m+e_F+n_F-b_F}^{b_c+b_F}\bigr), & \text{if } \Sigma_F \text{ is orientable}.
\end{cases}
\]
\end{lemma}

\begin{proof}
Suppose first that $g\geqslant 3$. From \cite[Corollary 1.5]{Earle} it follows that the branched covering map $p\colon N_g\to \Sigma_F$ has the Birman--Hilden property; see also \cite[Corollary~8.6 $\&$ end of
p.~20]{Zieschang73}.  Since the set of marked points in $N_g$ is $X$ and $Y=p(X)\subset \Sigma_F$, we obtain a short exact sequence 
\[
1 \to F \to \operatorname{S}\modnn \to \operatorname{L}\Gamma_F^*(Y) \to 1,
\]
where $\operatorname{S}\modnn$ is the subgroup of $\modnn$ consisting of
mapping classes admitting representatives that preserve the fibers of $p$, and
$\operatorname{L}\Gamma_F^*(Y)$ is the subgroup of $\Gamma_F^*(Y)$
consisting of mapping classes that lift under the branched cover $p$ and preserve the set $Y$.
Moreover, by \cite[Corollary~8.7 $\&$ end of p.~20]{Zieschang73} we have
$\operatorname{S}\modnn=NF$,  therefore
\[
WF=NF/F\cong \operatorname{L}\Gamma_F^*(Y).
\]
Then $WF$ is isomorphic to a subgroup of  $\Gamma_F^*(Y)$. Since,  the latter  has a finite index subgroup of the mapping class group $\Mod(\Sigma_{g_F,b_m+e_F+n_F-b_F}^{b_c+b_F})$ defined in \Cref{eq:ModUnderlying}, the inequality
for its $\vcd$ follows.

Now suppose that $g=1,2$. 

Since the action of $F$ preserves $X$ setwise, we have $X=p^{-1}(Y)$. Indeed, if $x\in X$, then the fiber $p^{-1}(p(x))$ is the $F$-orbit of $x$, which is contained in $X$. Consider the restriction $\widehat{p}:=p|_{N_g-X}\colon N_g-X\to \Sigma_F-Y$. This is a regular branched covering, and the surface $N_g-X$ is hyperbolic since $g+n-2>0$. Thus, by \cite[Corollary~8.6 $\&$ end of p.~20]{Zieschang73} $\widehat{p}$ satisfies the Birman--Hilden property, and the result follows by the same argument as above.

\end{proof}

Using the previous lemmas  we obtain the following results, where we take \begin{equation}\label{eq:ModUnderlying2}
\Mod\big(\Sigma_{g_F,b_m+e_F}^{b_c}\big)=
\begin{cases}
\calN_{g_F,b_m+e_F}^{b_c}, & \text{if } \Sigma_F \text{ is non-orientable},\\
\Gamma_{g_F,b_m+e_F}^{b_c}, & \text{if } \Sigma_F \text{ is orientable}.
\end{cases}
\end{equation}

\begin{proposition}\label{Prop:Bound:WF}
Let $g+n-2>0$ and consider a finite subgroup  $F$ of $\modnn$. Set
\[
\epsilon_F=
\begin{cases}
0, & \text{if } b_m+e_F+b_c>0,\\
0, & \text{if } \Sigma_F \text{ is orientable and } g_F\in\{0,1\},\\
0, & \text{if } \Sigma_F \text{ is non-orientable and } g_F\in\{1,2\},\\
1, & \text{otherwise}.
\end{cases}
\]
Then

\[
\vcd(WF) \leqslant \vcd\big(\Mod\big(\Sigma_{g_F,b_m+e_F}^{b_c}\big)\big)
+\min \left\{ \frac{2n}{|F|}, \frac{n}{|F|}+\frac{b_m}{2} \right\}  
+ \epsilon_F.
\]

\end{proposition}

\begin{proof} 
Let  $F$ be a finite subgroup of $\modnn$. From \Cref{lem:vcd:Weyl:Group} and  \Cref{Lem:Adding:Marked:Points}(5) we obtain
\[
    \vcd(WF) \leqslant \vcd\big(\Mod\big(\Sigma_{g_F,b_m+e_F+n_F-b_F}^{b_c+b_F}\big)\big)\leqslant \vcd\big(\Mod\big(\Sigma_{g_F,b_m+e_F+n_F}^{b_c}\big)\big) +b_F.
\]

Set $r=b_m+e_F$, $b=b_c$, and $s=n_F$.   If $r+b>0$, then
\Cref{Lem:Adding:Marked:Points}(1) shows that adding the $n_F$ free marked points
increases the $\vcd$ by at most $n_F$. If $r=b=0$, part (2) of
the same lemma gives an upper bound of $n_F+1$. However, by parts (3) and (4) of \Cref{Lem:Adding:Marked:Points}, the additional unit
is not needed when $\Sigma_F$ is orientable of genus $0$ or $1$, or when $\Sigma_F$ is
non-orientable of genus $1$ or $2$. These are precisely the cases encoded
by $\epsilon_F$. Therefore
\[
\vcd\big(\Mod\big(\Sigma_{g_F,b_m+e_F+n_F}^{b_c}\big)\big) \leqslant \vcd\big(\Mod\big(\Sigma_{g_F,b_m+e_F}^{b_c}\big)\big) + n_F  + \epsilon_F,
\]
and hence,
\[
    \vcd(WF) \leqslant \vcd\big(\Mod\big(\Sigma_{g_F,b_m+e_F}^{b_c}\big)\big) +n_F+b_F+\epsilon_F.
\]
As each free $F$-orbit has cardinality $|F|$ and each $F$-orbit of a mirror point has cardinality $|F|/2$, the definitions of $n_F$ and $b_F$ yield the  inequality  $|F|\, n_F+(|F|/2)\, b_F \leq n $. Consequently, 
\[
    n_F+b_F\leqslant \frac{2n}{|F|} \qquad\text{and}\qquad n_F+b_F \leqslant\frac{n}{|F|}+\frac{b_F}{2} \leqslant\frac{n}{|F|}+\frac{b_m}{2}.
\]
The desired inequality follows from the previous bound for $\vcd(WF)$.
\end{proof}

\begin{corollary}\label{Bound:WF:by:OF} 
Let $F$ be a non-trivial finite subgroup of $\modnn$, and let  $g\geqslant 3$.
If
\begin{itemize}
    \item there exists $k\geqslant 0$ such that $\vcd\big(\Mod\big(\Sigma_{g_F,b_m+e_F}^{b_c}\big)\big) + \lambda(F) \leqslant \vcd(\calN_g) + k$, and
    \item  $k+\epsilon_F + \min\left\{ \frac{2n}{|F|}, \frac{n}{|F|}+\frac{b_m}{2} \right\} \leqslant n+1$ with  $\epsilon_F$ defined as in \Cref{Prop:Bound:WF}, then
\end{itemize}
\[
\vcd(WF) + \lambda(F) \leqslant \vcd(\modnn).
\]
\end{corollary}
\begin{proof}
By \Cref{Prop:Bound:WF} and the hypotheses we have
\begin{align*}
\vcd(WF)+\lambda(F) 
&\leqslant \vcd\big(\Mod\big(\Sigma_{g_F,b_m+e_F}^{b_c}\big)\big) + \min\left\{ \frac{2n}{|F|}, \frac{n}{|F|}+\frac{b_m}{2} \right\} + \epsilon_F + \lambda(F) &\\
&\leqslant \vcd(\calN_g) + k + \min\left\{ \frac{2n}{|F|}, \frac{n}{|F|}+\frac{b_m}{2} \right\} + \epsilon_F &\\
&\leqslant \vcd(\calN_g) + n + 1 \ \leqslant \vcd(\modnn).\\
\end{align*}
\end{proof}

\subsection{Other auxiliary lemmas for finite group actions on non-orientable surfaces}
For some of our computations in \Cref{sec:GenusOne,sec:GenusTwo}, it will be useful to identify finite subgroups  of  diffeomorphisms of a non-orientable surface with finite subgroups of diffeomorphisms of its orientable double cover. 

\begin{lemma}\label{lemma:DoubleCoverGroups}
Let $n\geqslant  0$ and $g\geqslant  1$. Consider any finite subgroup $F$ of  diffeomorphisms of $N_{g,n}$. Then $F$ is isomorphic to a finite subgroup of orientation-preserving diffeomorphisms of $S_{g-1,2n}$ whose elements commute with the deck transformation of the orientable double cover.    
\end{lemma}
\begin{proof}
By lifting the action of $F$ to the orientable double cover $\pi:S_{g-1,2n} \to N_{g,n}$, we obtain a finite group $\bar F$  of diffeomorphisms of $S_{g-1,2n}$ that fits in the short exact sequence\[1\to D\to \bar F \to F  \to 1,\] where $D$ is the group of deck transformations of $\pi$. Recall that $D$ is generated by an orientation reversing involution of $S_{g-1,2n}$ without fixed points. Every lift normalizes $D$, and $D$ has order two, then all the elements of $\bar F$ commute with the deck transformation. Denote by $\bar F_0$ the index $2$ subgroup of $\bar F$ consisting of all elements that preserve orientation. Since $D\cap \bar F_0$ is trivial, we get that $\bar F_0$ projects isomorphically onto $F$.       
\end{proof}

On the other hand, for $g\geqslant 3$ and $n\geqslant 1$, we have the {\it Birman exact sequence}  
\begin{equation}\label{eq:ses-full}
1 \to B_{n}(N_{g}) \to \modnn 
   \xrightarrow{\varphi} \calN_g \to 1,
\end{equation}
by \cite[Theorem 2.1]{KOR} and \cite[Théorème 1]{Gramain}. Here  $\varphi:\modnn\to\calN_g$ is the {\it forgetful} homomorphism and  $B_{n}(N_{g})$ is the {\it full} braid group on $n$ strands on $N_{g}$. Since $B_{n}(N_{g})$ is torsion-free, we may regard any finite subgroup $F$ of $\modnn$ as a finite subgroup of $\calN_g$. This will be useful in  \Cref{sec:GENUS4&5,sec:GENUS6}.

\begin{lemma}\label{lem:COMPARE} Let $g\geqslant 3$, $n\geqslant 1$, and consider the forgetful homomorphism $\varphi:\modnn\to\calN_g$.  Then for any finite subgroup $F$ of $\modnn$, the  subgroup  $\varphi(F)$  of $\calN_g$ is isomorphic to $F$.     
\end{lemma}

\section{Genus 1}\label{sec:GenusOne}
In this section we discuss the case $g=1$ of our main result and prove \Cref{Thm: MainProjectivePlane}.

\begin{proposition}\label{propo: vcdOrbifoldProjectivePlane}
Let $n\geqslant 1$  and  let $F$ be a non-trivial finite subgroup of  diffeomorphisms of $N_{1,n}$. Then
\begin{itemize}
    \item[(a)] $\vcd\big(\Mod\big(\Sigma_{g_F,b_m+e_F}^{b_c}\big)\big)=0$, and
    \item[(b)]  $|F|\leqslant \max\{4n,60\}$ for $n\geqslant 3$.
\end{itemize}
\end{proposition}

\begin{proof}
Let $F$ be a non-trivial finite subgroup of diffeomorphisms of $N_{1,n}$. By capping the punctures, the action of $F$ on $N_{1,n}$ induces an action of $F$ on the closed projective plane $N_1=\mathbb{RP}^2$, and we may consider the quotient orbifold $O_F=N_1/F$.  By lifting this action to the universal covering, we obtain a group $\tilde F$ acting by diffeomorphisms on the sphere $S^2$ which is an extension of $F$ by the fundamental group of the projective plane, and such that $S^2/\tilde F$ is diffeomorphic to $O_F$. Hence $O_F$ is a spherical orbifold. We list the possible spherical orbifolds in \Cref{table:projective-plane-orbifolds}, together with its signature and the corresponding invariants. The last column gives the $\vcd$ of $\Mod\big(\Sigma_{g_F,b_m+e_F}^{b_c}\big)$, and the first statement follows.

{\scriptsize
\begin{table}[h]
\centering
\caption{Possible quotient orbifolds of finite group actions on $N_1=\mathbb{RP}^2$, see for instance \cite{Conway}. }
\label{table:projective-plane-orbifolds}
\begin{tabular}{|c|c|c|c|c|c|c|c|}
\hline
 &  &  &  &  &  & &  \\
Signature & $\Sigma_F$ & $g_F$ & $b_c$ & $b_m$ & $e_F$ & $\Mod\big(\Sigma_{g_F,b_m+e_F}^{b_c}\big)$ & $\vcd$ \\
 &  &  &  &  &  & &  \\
\hline
$(0;+;[-];\{-\})$ & $S^2$ & $0$ & $0$ & $0$ & $0$ & $\Gamma^{0}_{0,0}$ & $0$ \\
\hline
$(0;+;[q,q];\{-\})$ & $S^2$ & $0$ & $0$ & $0$ & $2$ & $\Gamma^{0}_{0,2}$ & $0$ \\
\hline
$(0;+;[2,2,q];\{-\})$ & $S^2$ & $0$ & $0$ & $0$ & $3$ & $\Gamma^{0}_{0,3}$ & $0$ \\
\hline
$(0;+;[2,3,3];\{-\})$ & $S^2$ & $0$ & $0$ & $0$ & $3$ & $\Gamma^{0}_{0,3}$ & $0$ \\
\hline
$(0;+;[2,3,4];\{-\})$ & $S^2$ & $0$ & $0$ & $0$ & $3$ & $\Gamma^{0}_{0,3}$ & $0$ \\
\hline
$(0;+;[2,3,5];\{-\})$ & $S^2$ & $0$ & $0$ & $0$ & $3$ & $\Gamma^{0}_{0,3}$ & $0$ \\
\hline
$(0;+;[-];\{(-)\})$ & $D^2$ & $0$ & $0$ & $1$ & $0$ & $\Gamma^{0}_{0,1}$ & $0$ \\
\hline
$(0;+;[-];\{(q,q)\})$ & $D^2$ & $0$ & $1$ & $0$ & $0$ & $\Gamma^{1}_{0,0}$ & $0$ \\
\hline
$(0;+;[-];\{(2,2,q)\})$ & $D^2$ & $0$ & $1$ & $0$ & $0$ & $\Gamma^{1}_{0,0}$ & $0$ \\
\hline
$(0;+;[-];\{(2,3,3)\})$ & $D^2$ & $0$ & $1$ & $0$ & $0$ & $\Gamma^{1}_{0,0}$ & $0$ \\
\hline
$(0;+;[-];\{(2,3,4)\})$ & $D^2$ & $0$ & $1$ & $0$ & $0$ & $\Gamma^{1}_{0,0}$ & $0$ \\
\hline
$(0;+;[-];\{(2,3,5)\})$ & $D^2$ & $0$ & $1$ & $0$ & $0$ & $\Gamma^{1}_{0,0}$ & $0$ \\
\hline
$(0;+;[q];\{(-)\})$ & $D^2$ & $0$ & $0$ & $1$ & $1$ & $\Gamma^{0}_{0,2}$ & $0$ \\
\hline
$(0;+;[2];\{(q)\})$ & $D^2$ & $0$ & $1$ & $0$ & $1$ & $\Gamma^{1}_{0,1}$ & $0$ \\
\hline
$(0;+;[3];\{(2)\})$ & $D^2$ & $0$ & $1$ & $0$ & $1$ & $\Gamma^{1}_{0,1}$ & $0$ \\
\hline
$(1;-;[-];\{-\})$ & $\mathbb{RP}^2$ & $1$ & $0$ & $0$ & $0$ & $\calN^{0}_{1,0}$ & $0$ \\
\hline
$(1;-;[q];\{-\})$ & $\mathbb{RP}^2$ & $1$ & $0$ & $0$ & $1$ & $\calN^{0}_{1,1}$ & $0$ \\
\hline
\end{tabular}
\end{table}}

Let us prove the second statement. By \Cref{lemma:DoubleCoverGroups}, the group $F$ is isomorphic to a finite subgroup of orientation-preserving diffeomorphisms of $S^2$ with $2n$ marked points. From \cite[Theorem 3]{stukow} it follows that  $F$ must be isomorphic to a subgroup of either $C_{2n-1}$, the dihedral group  $D_{2n}$ of order $4n$, $D_{2n-2}$, the alternating group $A_4$, the symmetric group $S_4$ or $A_5$. In the cyclic and dihedral cases, $|F|\leqslant 4n$, and  in the remaining cases $|F|\leqslant 60$. Hence $|F|\leqslant \max\{4n,60\}$.

\end{proof}

\begin{proposition}\label{prop:upperbound:ProjectivePlane}
Let $n\geqslant 3$ and let $F$ be a non-trivial finite subgroup of $\calN_{1,n}$.
Then
\begin{itemize}
    \item[(a)] $\vcd(WF)+\lambda(F)\leqslant \vcd(\calN_{1,n})=n-2$ for  $n\geqslant 6$,
    \item[(b)] $\vcd(WF)+\lambda(F)\leqslant 5$ for $3\leqslant n \leqslant 5$.
\end{itemize}
\end{proposition}

\begin{proof}
Let $F$ be a finite non-trivial subgroup of $\calN_{1,n}$. By Nielsen
realization, $F$ can be realized as a finite subgroup of diffeomorphisms of
$N_{1,n}$, and we consider the quotient orbifold $O_F=N_{1,n}/F$. By the
classification of the possible quotient orbifolds in
 \Cref{table:projective-plane-orbifolds} and \Cref{propo: vcdOrbifoldProjectivePlane}(a), we have $\epsilon_F=0$, $b_m\leqslant 1$  and
\[
\vcd\big(\Mod(\Sigma_{g_F,b_m+e_F}^{b_c})\big)=0.
\]

It follows from \Cref{Prop:Bound:WF} and
\Cref{lem:UpperBoundLength} that
\[
\vcd(WF)+\lambda(F)
\leqslant \min\left\{\frac{2n}{|F|},\frac{n}{|F|}+\frac{1}{2}\right\}
+\Omega(|F|).
\]
Let us denote the latter upper bound by $B_n(|F|)$. On the other hand, by \Cref{propo: vcdOrbifoldProjectivePlane}(b), we also have
$2\leqslant |F|\leqslant \max\{4n,60\}$.

Suppose first that $n\geqslant 18$. The function
$x\mapsto \frac{x}{2}-\frac52-\log_2(4x)$ is increasing for $x\geqslant18$
and is positive at $x=18$. Hence
$\log_2(4n)\leqslant \frac n2-\frac52$. Since $|F|\geqslant2$ and
$|F|\leqslant4n$, it follows that
\[
B_n(|F|)\leqslant \frac{n}{2}+\frac12+\Omega(|F|)
\leqslant \frac{n}{2}+\frac12+\log_2(4n)
\leqslant n-2.
\]
Hence $\vcd(WF)+\lambda(F)\leqslant n-2$.

For $7\leqslant n\leqslant17$, the program in
 Appendix \ref{Appendix:code:projective-plane} verifies that
$B_n(|F|)<n-1$ for every integer $2\leqslant |F|\leqslant\max\{4n,60\}$.
Since $\vcd(WF)+\lambda(F)$ is an integer, the desired bound follows.

For $3\leqslant n\leqslant5$, the same program gives
\[
\max_{2\leqslant |F|\leqslant60}B_n(|F|)
=\begin{cases}
\frac{83}{16}, & n=3,\\
\frac{21}{4}, & n=4,\\
\frac{85}{16}, & n=5,
\end{cases}
%<6
\]
which is strictly less than $6$, and therefore $\vcd(WF)+\lambda(F)\leqslant5$ in these cases.

Finally, suppose that $n=6$. By \cite[Theorem 6]{GG12}, $F$ is
isomorphic to a subgroup of $D_{12}$, $D_{10}$, $S_4$, or $A_5$.
Therefore $|F|$ divides one of $24$, $20$, $24$, or $60$. Checking these
divisors gives $B_6(|F|)\leqslant9/2<5$, and hence
$\vcd(WF)+\lambda(F)\leqslant4=\vcd(\calN_{1,6})$.\end{proof}

\begin{remark}\label{rmk:FailsGenus1}
    For $n=4,5$, the inequality $\vcd(WF)+\lambda(F)\leqslant\vcd(\mathcal{N}_{1,n})$ can fail.  For instance, that is the case for a subgroup $F\cong D_8$ of order $16$. To realize this subgroup, consider a regular octagon with vertices on the equator of $S^2$. The rotation through $\pi/4$ about the vertical axis and the reflection in the vertical plane through two opposite vertices generate $D_8$. Both commute with the antipodal map, and their action descends to $\mathbb{RP}^2$. Mark the four images of the vertices and, for $n=5$ consider also the common image of the poles. Since $\lambda(D_8)=4$, we have $\vcd(WF)+\lambda(F)\geqslant 4>n-2=\vcd(\mathcal{N}_{1,n})$.
\end{remark}

From \Cref{prop:upperbound:ProjectivePlane} and \Cref{Conchita:criterion} we
obtain the following estimates for the proper cohomological and geometric
dimensions of $\calN_{1,n}$.

\begin{theorem}\label{Thm: MainProjectivePlane}
For $n\leqslant 3$ or $n\geqslant 6$,
\[
\vcd(\calN_{1,n})=\underline{\cd}(\calN_{1,n})=\underline{\gd}(\calN_{1,n}).
\]
Moreover, for $4\leqslant n \leqslant 5$ we have
\[
\vcd(\calN_{1,n})\leqslant \underline{\cd}(\calN_{1,n})
\leqslant \underline{\gd}(\calN_{1,n})
\leqslant 5.
\]
\end{theorem}

\begin{proof}
For $n\in \{0,1,2\}$ the group $\calN_{1,n}$ is finite and all the dimensions are 0. For $n=3$, the group $\calN_{1,n}$ is virtually free, and therefore $\vcd(\calN_{1,n})=\underline{\cd}(\calN_{1,n})=\underline{\gd}(\calN_{1,n})=1$.

If $n\geqslant 6$, then \Cref{prop:upperbound:ProjectivePlane}(a) and
\Cref{Conchita:criterion} imply
$\underline{\cd}(\calN_{1,n})=\vcd(\calN_{1,n})=n-2$. Since $n-2\geqslant 4$, the
inequality $\underline{\gd}(G)\leqslant \max\{3,\underline{\cd}(G)\}$ from \Cref{eq:InDim}
gives
$\underline{\gd}(\calN_{1,n})=\underline{\cd}(\calN_{1,n})$.

For $4\leqslant n\leqslant 5$, the upper bound from
\Cref{prop:upperbound:ProjectivePlane}(b) and \Cref{Conchita:criterion} give
$\underline{\cd}(\calN_{1,n})\leqslant 5$, and the stated lower and upper bounds for
$\underline{\gd}$ follow again from \Cref{eq:InDim}.
\end{proof}

\section{Genus \texorpdfstring{$2$}{2}}\label{sec:GenusTwo}
In this section we prove \Cref{Thm: MainKB}, which gives the case $g=2$ of our main theorem.

\begin{proposition}\label{propo: vcdOrbifoldKleinBottle}
    Let $n\geqslant 1$, and let $F$ be a non-trivial finite subgroup of  diffeomorphisms of $N_{2,n}$. Then $\vcd\big(\Mod\big(\Sigma_{g_F,b_m+e_F}^{b_c}\big)\big)=0$.
\end{proposition}
\begin{proof}
{\scriptsize 
\begin{table}[h]
\centering
\caption{\small Possible wallpaper groups $\tilde{F}$ that lift $F$}
\label{table:wallpapers}
\begin{tabular}{|c|c|c|c|c|c|c|c|}
\hline
& & & & & & & \\
Wallpaper group & Signature  & $g_F$ & $b_c$ & $b_m$ & $e_F$ & $\Mod\big(\Sigma_{g_F,b_m+e_F}^{b_c}\big)$  & $\vcd$ \\
& & & & & & & \\
\hline
$\mathbf{p6m}$ & $(0;+;[-];\{(6,3,2)\})$ & 0 & 1 & 0 & 0 & $\Gamma^{1}_{0,0}$ & 0 \\
\hline
$\mathbf{p4m}$ & $(0;+;[-];\{(4,4,2)\})$ & 0 & 1 & 0 & 0 & $\Gamma^{1}_{0,0}$ & 0 \\
\hline
$\mathbf{p4g}$ & $(0;+;[4];\{(2)\})$ & 0 & 1 & 0 & 1 & $\Gamma^{1}_{0,1}$ & 0 \\
\hline
$\mathbf{p3m1}$ & $(0;+;[-];\{(3,3,3)\})$ & 0 & 1 & 0 & 0 & $\Gamma^{1}_{0,0}$ & 0 \\
\hline
$\mathbf{p31m}$ & $(0;+;[3];\{(3)\})$ & 0 & 1 & 0 & 1 & $\Gamma^{1}_{0,1}$ & 0 \\
\hline
$\mathbf{pmm}$ & $(0;+;[-];\{(2,2,2,2)\})$ & 0 & 1 & 0 & 0 & $\Gamma^{1}_{0,0}$ & 0 \\
\hline
$\mathbf{cmm}$ & $(0;+;[2];\{(2,2)\})$ & 0 & 1 & 0 & 1 & $\Gamma^{1}_{0,1}$ & 0 \\
\hline
$\mathbf{pmg}$ & $(0;+;[2,2];\{(-)\})$ & 0 & 0 & 1 & 2 & $\Gamma^{0}_{0,3}$ & 0 \\
\hline
$\mathbf{pgg}$ & $(1;-;[2,2];\{-\})$ & 1 & 0 & 0 & 2 & $\calN^{0}_{1,2}$ & 0 \\
\hline
$\mathbf{pm}$ & $(0;+;[-];\{(-),(-)\})$ & 0 & 0 & 2 & 0 & $\Gamma^{0}_{0,2}$ & 0 \\
\hline
$\mathbf{cm}$ & $(1;-;[-];\{(-)\})$ & 1 & 0 & 1 & 0 & $\calN^{0}_{1,1}$ & 0 \\
\hline
$\mathbf{pg}$ & $(2;-;[-];\{-\})$ & 2 & 0 & 0 & 0 & $\calN^{0}_{2,0}$ & 0 \\
\hline
\end{tabular}
\end{table}}

Let $F$ be a finite subgroup of diffeomorphisms of $N_{2,n}$. By capping the punctures, the action of $F$ on $N_{2,n}$ induces an action of $F$ on the closed Klein bottle $N_2$, and we may consider the quotient orbifold $O_F=N_2/F$. The uniformization theorem implies that there is a metric of constant curvature $0$ on $N_2$ such that the action of $F$ is by isometries. By lifting this action to the universal covering, we obtain a group $\tilde F$ acting by Euclidean isometries on $\mathbb R^2$ which is an extension of $F$ by the fundamental group of the Klein bottle, and such that $\mathbb{R}^2/\tilde F$ is diffeomorphic to $O_F$. Hence $\tilde F$ is a wallpaper group that contains an orientation-reversing isometry. We list the possible wallpaper groups $\tilde F$ with this property in \cref{table:wallpapers}, following \cite[Table~II]{Conway}, together with the signature of the associated quotient orbifold $O_F$ and the corresponding invariants of $O_F$. The last column gives the $\vcd$ of the mapping class group $\Mod\big(\Sigma_{g_F,b_m+e_F}^{b_c}\big)$ appearing in the previous column.
 \end{proof}

\begin{proposition}\label{prop:FiniteGroupKlein}
    Let $n\geqslant 1$, and let $F$ be a finite subgroup of $\calN_{2,n}$. Then $|F|\leqslant 4n$. 
    
\end{proposition}
\begin{proof}
By the Nielsen realization theorem, we may realize $F$ as a finite group of diffeomorphisms of $N_{2,n}$. We use the same notation for this realization and for the corresponding subgroup of the mapping class group.

Let $X$ be the set of $n$ marked points in $N_2$. Choose $x\in X$, let $F_x$ be its stabilizer, and set $X'=X\setminus\{x\}$. Since $F_x$ fixes $x$ and preserves $X'$, we regard it as a finite subgroup of $\Mod(N_{2,1};X')$. Since $N_{2,1}$ has negative Euler characteristic, by \cite[Theorem 2.1]{KOR} and \cite[Theorem 1.1]{Yagasaki} we have the {\it Birman exact sequence}
\[
    1 \to B_{n-1}(N_{2,1}) \to \Mod(N_{2,1};X') \xrightarrow[]{\varphi} \Mod(N_{2,1}) \to 1,
\]
where $\varphi$ `forgets' the marked points $X'$, and $B_{n-1}(N_{2,1})$ is the $n-1$-strand full braid group on $N_{2,1}$. Because $B_{n-1}(N_{2,1})$ is torsion-free and $F_x$ is finite, the restriction of $\varphi$ to $F_x$ is injective. By \cite[Theorem A.5]{Stu06}, $\Mod(N_{2,1})\cong D_\infty\times C_2$, whose finite subgroups have order at most $4$. Thus $|F_x|\leq 4$. Since $F\cdot x\subseteq X$, the orbit--stabilizer theorem gives
\[
    |F|=|F_x|\,|F\cdot x|\leqslant4n.
\]\end{proof}

\begin{proposition}\label{prop:upperbound:KB}
    Let $n\geqslant2$ and let $F$ be a non-trivial finite subgroup of $\calN_{2,n}$. Then
    \[
    \vcd(WF)+\lambda(F) \leqslant \vcd(\calN_{2,n})=n
    \qquad\text{for } n\geqslant 3,
    \]
    and 
    \[
    \vcd(WF)+\lambda(F) \leqslant 3=\vcd(\calN_{2,2})+1
    \qquad\text{for } n=2.
    \]
\end{proposition}
\begin{proof}
Let $F$ be a finite non-trivial subgroup of $\calN_{2,n}$. By Nielsen
realization, $F$ can be realized as a finite subgroup of diffeomorphisms of
$N_{2,n}$. By \Cref{lem:UpperBoundLength} and \Cref{prop:FiniteGroupKlein} we have 
\[
\lambda(F)\leqslant \log_2(|F|)\leqslant \log_2(4n)\leqslant \log_2(n)+2.
\]
Using \Cref{Prop:Bound:WF}, the fact that $\epsilon_F=0$ and $b_m\leq 2$ for all  the quotient
orbifolds $O_F$ in \cref{table:wallpapers}, and \Cref{propo: vcdOrbifoldKleinBottle},
we get
\[
\vcd(WF)+\lambda(F)
\leqslant \min \left\{ \frac{2n}{|F|}, \frac{n}{|F|}+1 \right\}+\lambda(F).
\]
Since $|F|\geqslant 2$, the previous logarithmic bound gives
\[
\vcd(WF)+\lambda(F)
\leqslant \frac{n}{2}+\log_2(n)+3.
\]
It is straightforward to verify that
$\frac{n}{2}+\log_2(n)+3\leqslant n$ for $n\geqslant 14$. Hence the desired
inequality holds for all $n\geqslant 14$.

Now consider $2\leqslant n\leqslant 13$. 
By
\Cref{prop:FiniteGroupKlein} and \Cref{lem:UpperBoundLength}, it suffices to
obtain an upper bound for
\[
\min\left\{\frac{2n}{k},\frac{n}{k}+1\right\}+\Omega(k)
\]
for every integer $k$ with $2\leqslant k\leqslant 4n$. The program in
Appendix~\ref{Appendix:code:n:F} verifies that the maximum of these quantities is strictly smaller
than $n+1$ when $3\leqslant n\leqslant 13$. And it is $7/2$ when $n=2$, which is strictly smaller that $4$. Since
$\vcd(WF)+\lambda(F)$ is an integer, these bounds give the stated upper
bounds.

\end{proof}

\begin{remark}\label{rmk:FailsGenus2}
For $n=2$, the inequality
$\vcd(WF)+\lambda(F)\leqslant\vcd(\mathcal{N}_{2,2})$
can fail. For instance, consider a subgroup
$F\cong C_2\times C_2\times C_2$ of order $8$, realized as follows.
Let $C=\{(x,y,z)\in\mathbb{R}^3:x^2+y^2=1,\,-1\leqslant z\leqslant 1\}$
be a cylinder, and identify antipodal points on each boundary circle.
The resulting surface is a Klein bottle. Mark the two points represented
by $(0,1,1)$ and $(0,1,-1)$. The rotation through $\pi$ about the
$z$-axis and the reflections across the $xy$- and $yz$-planes descend
to three commuting involutions of the Klein bottle that preserve
the set of marked points. The eight homeomorphisms obtained by composing these involutions,
including the identity, induce pairwise distinct automorphisms of
the first homology group of the Klein bottle with the two marked
points removed. Hence they represent distinct mapping classes
and generate a subgroup
$F\cong C_2\times C_2\times C_2$ of $\mathcal{N}_{2,2}$. 
Since $\lambda(F)=3$, we have
$\vcd(WF)+\lambda(F)\geqslant 3>2=\vcd(\mathcal{N}_{2,2})$.
\end{remark}

From \Cref{prop:upperbound:KB} and \Cref{Conchita:criterion} we obtain the following estimates for the proper cohomological and geometric dimensions of $\calN_{2,n}$.

\begin{theorem}\label{Thm: MainKB}
For $n=1$ or $n\geqslant 3$, we have
\[
\vcd(\calN_{2,n})=\underline{\cd}(\calN_{2,n})=\underline{\gd}(\calN_{2,n}).
\]
For $n=2$ we have
\[
\vcd(\calN_{2,2})\leqslant \underline{\cd}(\calN_{2,2})\leqslant \underline{\gd}(\calN_{2,2})\leqslant \vcd(\calN_{2,2})+1.
\]
\end{theorem}

\begin{proof}
The case $n=1$ follows from \cite[Proof of Theorem~1.4]{HSSTN24}, since for one marked point the {\it pure} and {\it full} mapping class groups agree. 

Suppose now that $n\geqslant 2$. By  \Cref{Conchita:criterion}, the inequalities in \Cref{prop:upperbound:KB} give the corresponding upper bounds for $\underline{\cd}(\calN_{2,n})$ for every non-trivial finite subgroup.

On the other hand, recall that $\underline{\gd}(G)\leqslant \max\{3,\underline{\cd}(G)\}$ for every group $G$ by \Cref{eq:InDim}. For $n\geqslant 3$, we have that $\underline{\cd}(\calN_{2,n})\geqslant \vcd(\calN_{2,n})=n\geqslant 3$ and hence $\underline{\gd}(\calN_{2,n})=\underline{\cd}(\calN_{2,n})$. For $n=2$, the previous paragraph gives $\underline{\cd}(\calN_{2,2})\leqslant 3=\vcd(\calN_{2,2})+1$, and \Cref{eq:InDim} yields the stated upper bound for $\underline{\gd}$.
\end{proof}

\section{Genus 3}\label{sec:genus-3}

In this section we prove in \Cref{prop:N3:marked:points} that the proper cohomological and geometric dimensions of
$\calN_{3,n}$ coincide with its $\vcd$ for
$n\geqslant 1$. 

According to \cite[Proposition 1]{BEM14} the finite groups $F$ of diffeomorphisms acting  on a non-orientable surface $N_3$ of genus $3$ are $C_2$, $C_3$, $C_4$, $C_6$, $C_2\times C_2$, $D_3$, $D_4$ and $D_6$. The signature of the orbifold $O_F=N_3/F$  \[\big(g_F;\pm;[m_1,m_2,\ldots,m_{e_F}]; \{ (n_{i,1},\ldots, n_{i,s_i}) , i=1,\ldots,b\}\big),\]
must satisfy the Riemann--Hurwitz equation (see for instance \cite[Section 2.4]{HSSTN24}):
\begin{equation}\label{eq:RH}
    \dfrac{1}{|F|} = \eta g_F + b-2+\sum_{i=1}^{e_F} \left(1-\frac{1}{m_i}\right)+\frac{1}{2}\sum_{i=1}^{b_c}\sum_{j=1}^{s_i} \left( 1-\frac{1}{n_{i,j}}\right),
\end{equation}
where $\eta=2$ if the signature has a $+$ and $\eta=1$ otherwise. The code in
Appendix~\ref{Appendix:code:genus-3} enumerates, for each possible group order, all the
orbifold signatures satisfying \Cref{eq:RH}; these are the rows of
\Cref{table:orbifolds-genus-3}. 
\begin{proposition}\label{prop:orbifolds-genus-3}
The  signature of every quotient orbifold $O_F$ arising from an action of a finite group
$F$ on $N_3$ appears in \Cref{table:orbifolds-genus-3}. 
\end{proposition}

\begin{remark}
  Since the enumeration listed in \Cref{table:orbifolds-genus-3} does not construct a
surface-kernel epimorphism for each row, realizability of every candidate is
not claimed.  \end{remark}

The last columns of \Cref{table:orbifolds-genus-3} contain the
numerical data needed to apply \Cref{Bound:WF:by:OF} to the marked surface
$N_{3,n}$. For each row  we list
\[
    k:=\vcd(\Mod\big(\Sigma_{g_F,b_m+e_F}^{b_c}\big))+\lambda(F)-\vcd(\calN_3) = \vcd(\Mod\big(\Sigma_{g_F,b_m+e_F}^{b_c}\big))+\lambda(F)-1.
\]
Every quotient orbifold listed in the table has elliptic points or boundary components, so $\epsilon_F=0$ by \Cref{Prop:Bound:WF}. The numerical hypothesis of \Cref{Bound:WF:by:OF} is satisfied whenever $k+n/|F|+b_m/2\leqslant n+1$,
and the last column records the numerical integer threshold from \Cref{Bound:WF:by:OF}
\[
    n_0:=\max\left\{0, \left\lceil \frac{|F|}{|F|-1}\left(k+\frac{b_m}{2}-1\right) \right\rceil\right\}.
\]

{\scriptsize
\begin{table}[h]
\centering
\caption{\small Riemann--Hurwitz candidates for possible signatures of quotient orbifolds of finite group actions on $N_3$.}
\label{table:orbifolds-genus-3}
\renewcommand{\arraystretch}{1.15}
\begin{tabular}{|c|c|c|c|c|c|c|c|c|c|}
\hline
& & & & & & & & & \\
Group $F$
& $|F|$
& Signature
& $b_m$
& $\Mod\big(\Sigma_{g_F,b_m+e_F}^{b_c}\big)$
& $\vcd$
& $\lambda(F)\leqslant$
& $k$
& $\vcd+\lambda(F)\leqslant$
& $n_0$\\
& & & & & & & & & \\
\hline
$C_2$ & $2$ & $(0;+;[2,2,2,2,2];\{-\})$ & $0$ & $\Gamma_{0,5}^{0}$ & $2$ & $1$ & $2$ & $3$ & $2$ \\
\hline
$C_2$ & $2$ & $(0;+;[2,2,2];\{(-)\})$ & $1$ & $\Gamma_{0,4}^{0}$ & $1$ & $1$ & $1$ & $2$ & $1$ \\
\hline
$C_2$ & $2$ & $(0;+;[2];\{(-),(-)\})$ & $2$ & $\Gamma_{0,3}^{0}$ & $0$ & $1$ & $0$ & $1$ & $0$ \\
\hline
$C_2$ & $2$ & $(1;+;[2];\{-\})$ & $0$ & $\Gamma_{1,1}^{0}$ & $1$ & $1$ & $1$ & $2$ & $0$ \\
\hline
$C_2$ & $2$ & $(1;-;[2,2,2];\{-\})$ & $0$ & $\mathcal{N}_{1,3}^{0}$ & $1$ & $1$ & $1$ & $2$ & $0$ \\
\hline
$C_2$ & $2$ & $(1;-;[2];\{(-)\})$ & $1$ & $\mathcal{N}_{1,2}^{0}$ & $0$ & $1$ & $0$ & $1$ & $0$ \\
\hline
$C_2$ & $2$ & $(2;-;[2];\{-\})$ & $0$ & $\mathcal{N}_{2,1}^{0}$ & $1$ & $1$ & $1$ & $2$ & $0$ \\
\hline
$C_3$ & $3$ & $(1;-;[3,3];\{-\})$ & $0$ & $\mathcal{N}_{1,2}^{0}$ & $0$ & $1$ & $0$ & $1$ & $0$ \\
\hline
$C_4$ & $4$ & $(0;+;[4,4,4];\{-\})$ & $0$ & $\Gamma_{0,3}^{0}$ & $0$ & $2$ & $1$ & $2$ & $0$ \\
\hline
$C_4$ & $4$ & $(0;+;[2,2,2,4];\{-\})$ & $0$ & $\Gamma_{0,4}^{0}$ & $1$ & $2$ & $2$ & $3$ & $2$ \\
\hline
$C_4$ & $4$ & $(0;+;[2,4];\{(-)\})$ & $1$ & $\Gamma_{0,3}^{0}$ & $0$ & $2$ & $1$ & $2$ & $1$ \\
\hline
$C_4$ & $4$ & $(1;-;[2,4];\{-\})$ & $0$ & $\mathcal{N}_{1,2}^{0}$ & $0$ & $2$ & $1$ & $2$ & $0$ \\
\hline
$C_6$ & $6$ & $(0;+;[3,3,6];\{-\})$ & $0$ & $\Gamma_{0,3}^{0}$ & $0$ & $2$ & $1$ & $2$ & $0$ \\
\hline
$C_6$ & $6$ & $(0;+;[2,6,6];\{-\})$ & $0$ & $\Gamma_{0,3}^{0}$ & $0$ & $2$ & $1$ & $2$ & $0$ \\
\hline
$C_6$ & $6$ & $(0;+;[2,2,2,3];\{-\})$ & $0$ & $\Gamma_{0,4}^{0}$ & $1$ & $2$ & $2$ & $3$ & $2$ \\
\hline
$C_6$ & $6$ & $(0;+;[2,3];\{(-)\})$ & $1$ & $\Gamma_{0,3}^{0}$ & $0$ & $2$ & $1$ & $2$ & $1$ \\
\hline
$C_6$ & $6$ & $(1;-;[2,3];\{-\})$ & $0$ & $\mathcal{N}_{1,2}^{0}$ & $0$ & $2$ & $1$ & $2$ & $0$ \\
\hline
$D_2$ & $4$ & $(0;+;[-];\{(2,2,2,2,2)\})$ & $0$ & $\Gamma_{0,0}^{1}$ & $0$ & $2$ & $1$ & $2$ & $0$ \\
\hline
$D_2$ & $4$ & $(0;+;[2];\{(2,2,2)\})$ & $0$ & $\Gamma_{0,1}^{1}$ & $0$ & $2$ & $1$ & $2$ & $0$ \\
\hline
$D_2$ & $4$ & $(0;+;[2,2];\{(2)\})$ & $0$ & $\Gamma_{0,2}^{1}$ & $1$ & $2$ & $2$ & $3$ & $2$ \\
\hline
$D_2$ & $4$ & $(0;+;[-];\{(-),(2)\})$ & $1$ & $\Gamma_{0,1}^{1}$ & $0$ & $2$ & $1$ & $2$ & $1$ \\
\hline
$D_2$ & $4$ & $(1;-;[-];\{(2)\})$ & $0$ & $\mathcal{N}_{1,0}^{1}$ & $0$ & $2$ & $1$ & $2$ & $0$ \\
\hline
$D_3$ & $6$ & $(0;+;[2,2,2,3];\{-\})$ & $0$ & $\Gamma_{0,4}^{0}$ & $1$ & $2$ & $2$ & $3$ & $2$ \\
\hline
$D_3$ & $6$ & $(0;+;[2];\{(3,3)\})$ & $0$ & $\Gamma_{0,1}^{1}$ & $0$ & $2$ & $1$ & $2$ & $0$ \\
\hline
$D_3$ & $6$ & $(0;+;[2,3];\{(-)\})$ & $1$ & $\Gamma_{0,3}^{0}$ & $0$ & $2$ & $1$ & $2$ & $1$ \\
\hline
$D_3$ & $6$ & $(1;-;[2,3];\{-\})$ & $0$ & $\mathcal{N}_{1,2}^{0}$ & $0$ & $2$ & $1$ & $2$ & $0$ \\
\hline
$D_4$ & $8$ & $(0;+;[-];\{(4,4,4)\})$ & $0$ & $\Gamma_{0,0}^{1}$ & $0$ & $3$ & $2$ & $3$ & $2$ \\
\hline
$D_4$ & $8$ & $(0;+;[-];\{(2,2,2,4)\})$ & $0$ & $\Gamma_{0,0}^{1}$ & $0$ & $3$ & $2$ & $3$ & $2$ \\
\hline
$D_4$ & $8$ & $(0;+;[4];\{(4)\})$ & $0$ & $\Gamma_{0,1}^{1}$ & $0$ & $3$ & $2$ & $3$ & $2$ \\
\hline
$D_4$ & $8$ & $(0;+;[2];\{(2,4)\})$ & $0$ & $\Gamma_{0,1}^{1}$ & $0$ & $3$ & $2$ & $3$ & $2$ \\
\hline
$D_6$ & $12$ & $(0;+;[-];\{(3,3,6)\})$ & $0$ & $\Gamma_{0,0}^{1}$ & $0$ & $3$ & $2$ & $3$ & $2$ \\
\hline
$D_6$ & $12$ & $(0;+;[-];\{(2,6,6)\})$ & $0$ & $\Gamma_{0,0}^{1}$ & $0$ & $3$ & $2$ & $3$ & $2$ \\
\hline
$D_6$ & $12$ & $(0;+;[-];\{(2,2,2,3)\})$ & $0$ & $\Gamma_{0,0}^{1}$ & $0$ & $3$ & $2$ & $3$ & $2$ \\
\hline
$D_6$ & $12$ & $(0;+;[6];\{(2)\})$ & $0$ & $\Gamma_{0,1}^{1}$ & $0$ & $3$ & $2$ & $3$ & $2$ \\
\hline
$D_6$ & $12$ & $(0;+;[3];\{(6)\})$ & $0$ & $\Gamma_{0,1}^{1}$ & $0$ & $3$ & $2$ & $3$ & $2$ \\
\hline
$D_6$ & $12$ & $(0;+;[2];\{(2,3)\})$ & $0$ & $\Gamma_{0,1}^{1}$ & $0$ & $3$ & $2$ & $3$ & $2$ \\
\hline
\end{tabular}
\end{table}
}

\begin{theorem}\label{prop:N3:marked:points}
Let $n\geqslant 1$. Then
\[
\underline{\gd}(\calN_{3,n})=\underline{\cd}(\calN_{3,n})=\vcd(\calN_{3,n}).
\]
\end{theorem}

\begin{proof}
First, notice that
$\vcd(\calN_{3,n})=n+2\geqslant 3$ for
$n\geqslant 1$. It then follows from \Cref{eq:InDim} that $\underline{\gd}(\calN_{3,n})=\underline{\cd}(\calN_{3,n})\geqslant \vcd(\calN_{3,n})$.
Hence we only need to prove that $\underline{\cd}(\calN_{3,n})\leqslant \vcd(\calN_{3,n})$.

For $n=1$, the {\it pure} and {\it full} mapping class groups agree; therefore, \cite[Corollary~1.3]{4Amigos_Spine_NO}
implies that $\underline{\gd}(\calN_{3,1})=\vcd(\calN_{3,1})$.
Assume from now on that $n\geqslant 2$.

Let $F$ be a finite subgroup of $\calN_{3,n}$. By Nielsen realization we can realize $F$ as a group of diffeomorphisms of $N_{3,n}$. By capping the punctures, the action of $F$ on $N_{3,n}$ induces an action of $F$ on the closed surface $N_3$, and we may consider the quotient orbifold $O_F=N_3/F$. By \Cref{prop:orbifolds-genus-3}
the signature of every quotient orbifold $O_F$ appears in \Cref{table:orbifolds-genus-3}. From the second-to-last column, we see that the following inequality holds
\[
\vcd\big(\Mod\big(\Sigma_{g_F,b_m+e_F}^{b_c}\big)\big)+
\lambda(F)\leqslant 1+k=
\vcd(\calN_3)+k,
\]
for every row, and from the last column of
\Cref{table:orbifolds-genus-3} we observe that  $n_0\leqslant 2\leqslant n$. Therefore the hypotheses of \Cref{Bound:WF:by:OF} are satisfied for every row and hence
\[
\vcd(WF)+\lambda(F)
\leqslant
\vcd(\calN_{3,n}).
\]
The conclusion of the theorem follows from \Cref{Conchita:criterion}.
\end{proof}

\section{Genus 4 and 5} \label{sec:GENUS4&5}

In this section, we prove \Cref{Thm:g:4:5:punctures}, which gives the cases $g=4,5$ of our main result.

First we state the following upper bounds, for the case of genus $g=4$, that can be obtained from the work in \cite{HSSTN24}.
\begin{proposition}\label{Prop:N4:HSSTN}
Let $F$ be a finite subgroup of $\calN_4$,  and let $O_F$ be the associated quotient orbifold with underlying surface $\Sigma_F$. 
\begin{enumerate}
    \item If $\Sigma_F$ is non-orientable or if $\Sigma_F$ is orientable with $g_F\geqslant 1$, then
    \[
    \vcd(\Gamma_F^*)+\lambda(F)\leqslant \vcd(\calN_4).
    \]

    \item If $\Sigma_F$ is orientable, $g_F=0$, and $e_F+b\geqslant 3$, or if $g_F=0$, $e_F+b\leqslant 3$, and $\lambda(F)\leqslant 3$, then
    \[
    \vcd(\Gamma_F^*)+\lambda(F)\leqslant \vcd(\calN_4)+1.
    \]

    \item If $\Sigma_F$ is orientable, $g_F=0$, $e_F+b<3$, and
    $\lambda(F)=4$, then
    \[
    \vcd(\Gamma_F^*)+\lambda(F)\leqslant \vcd(\calN_4)+1.
    \]
    The possible finite subgroups $F$ in this case are $C_2\times D_4$, $S_4$, and
    $C_2\times A_4$, of orders $16$, $24$, and $24$, respectively.

    \item If $\Sigma_F$ is orientable, $g_F=0$, $e_F+b<3$, and
    $\lambda(F)=5$, then
    \[
    \vcd(\Gamma_F^*)+\lambda(F)\leqslant \vcd(\calN_4)+2.
    \]
    The only possible finite subgroup $F$ in this case is $C_2\times S_4$, of order $48$.
\end{enumerate}
\end{proposition}

\begin{proof}
(1) Suppose that $\Sigma_F$ is non-orientable. By \cite[Theorem~4.2 and Theorem~4.3]{HSSTN24}, we have $\vcd(\Gamma_F^*)+\lambda(F)\leqslant \vcd(\calN_4)$, except possibly in the exceptional case described in \cite[Remark~4.3]{HSSTN24}, which corresponds to a group of order~$8$. However, by \cite[Table~4]{BEM14}, every group of order~$8$ acting on $N_4$ has an orientable quotient. Hence, this exceptional situation cannot occur when $\Sigma_F$ is non-orientable. If $\Sigma_F$ is orientable with $g_F\geqslant 1$, the result follows from \cite[Theorem~5.1]{BEM14}.

\medskip

(2) If $\Sigma_F$ is orientable, $g_F=0$, and $e_F+b\geqslant 3$, the result follows from \cite[Remark~5.3]{HSSTN24}. Suppose now that $\Sigma_F$ is orientable, $g_F=0$, and $e_F+b<3$. In this case $\vcd(\Gamma_F^*)\leqslant 2$. By \cite[Remark~5.5]{HSSTN24}, if $\vcd(\Gamma_F^*)=2$ then $\vcd(\Gamma_F^*)+\lambda(F)\leqslant 4=\vcd(\calN_4)+1$. If $\vcd(\Gamma_F^*)\leqslant 1$, the inequality follows since $\lambda(F)\leqslant 3$.

\medskip

(3) By the classification in \cite[Table~4]{BEM14}, the possible finite subgroups $F$ with
$\lambda(F)=4$ are $C_2\times D_4$, $S_4$, and $C_2\times A_4$. Their
quotient signatures have underlying mapping class group
$\Gamma_{0,0}^{1}$ or $\Gamma_{0,1}^{1}$, and hence
$\vcd(\Gamma_F^*)=0$. Therefore
\[
\vcd(\Gamma_F^*)+\lambda(F)
=4=\vcd(\calN_4)+1.
\]

\medskip

(4) The only possible finite subgroup $F$ with $\lambda(F)=5$ is $C_2\times S_4$, of order
$48$. The signature of the associated quotient orbifold $O_F$ is
\[
(0;+;[-];\{(2,4,6)\}),
\]
then the associated underlying mapping class group is $\Gamma_{0,0}^{1}$ and
$\vcd(\Gamma_F^*)=0$. Consequently,
\[
\vcd(\Gamma_F^*)+\lambda(F)
=5=\vcd(\calN_4)+2.
\]
\end{proof}

From \cite[Proof of Theorem 1.2]{HSSTN24} we have the following result for the genus $g=5$ case.

\begin{proposition}\label{Prop:N5:HSSTN}
Let $F$ be a finite subgroup of $\calN_5$, and let $O_F$ be the associated orbifold with underlying surface $\Sigma_F$. Then the following bounds hold.

\begin{enumerate}
    \item If $\Sigma_F$ is non-orientable, or if $\Sigma_F$ is orientable with $g_F\geqslant 1$, or $g_F=0$ and $e_F+b\geqslant 3$, then
    \[
    \vcd(\Gamma_F^*)+\lambda(F)\leqslant \vcd(\calN_5).
    \]

    \item If $\Sigma_F$ is orientable, $g_F=0$, and $e_F+b < 3$, then
    \[
    \vcd(\Gamma_F^*)+\lambda(F)\leqslant \vcd(\calN_5)+1.
    \]
\end{enumerate}
\end{proposition}

The following result is a consequence of \cite[Theorem 2.6 and its proof]{HSSTN24}. It shows that $\vcd(\Gamma_F^*)$ coincides with the $\vcd$ of a suitable surface mapping class group that appears in the hypothesis of \Cref{Bound:WF:by:OF}.

\begin{proposition}\label{thm:OrbifoldMCG}
Let $F$ be a finite subgroup of $\calN_{g}$. Then the orbifold mapping class group $\Gamma_F^*$ is commensurable with $\Mod(\Sigma_{g_F,b_m+e_F}^{b_c})$. In particular,
\[
\vcd(\Gamma_F^*)=
\begin{cases}
\vcd\bigl(\calN_{g_F,b_m+e_F}^{b_c}\bigr), & \text{if } \Sigma_F \text{ is non-orientable},\\
\vcd\bigl(\Gamma_{g_F,b_m+e_F}^{b_c}\bigr), & \text{if } \Sigma_F \text{ is orientable}.
\end{cases}
\]
\end{proposition}

\begin{theorem}\label{Thm:g:4:5:punctures}
Let $g\in\{4,5\}$ and $n\geqslant 1$. Then
\[
\underline{\gd}(\calN_{g,n})
=\underline{\cd}(\calN_{g,n})
=\vcd(\calN_{g,n}).
\]
\end{theorem}

\begin{proof} Let $g\in\{4,5\}$ and $n\geqslant 1$.
First, notice that
$\vcd(\calN_{g,n})=2g+n-4\geqslant 5$; hence, it follows from \Cref{eq:InDim} that
$\underline{\gd}(\calN_{g,n})=\underline{\cd}(\calN_{g,n})$.
Hence we only need to prove the second equality.

When $n=1$, the {\it pure} and {\it full} mapping class groups
agree, and from
\cite[Corollary~1.3]{4Amigos_Spine_NO}, we have
$\underline{\gd}(\calN_{g,1})=\vcd(\calN_{g,1})$.

Assume from now on that $n\geqslant 2$. By \Cref{Conchita:criterion}, to conclude $\underline{\cd}(\calN_{g,n})\leqslant \vcd(\calN_{g,n})$ it is enough to show that for every finite subgroup $F\leqslant \calN_{g,n}$ we have
\[
\vcd(WF)+\lambda(F)\leqslant \vcd(\calN_{g,n}).
\]

We next show that the hypotheses of \Cref{Bound:WF:by:OF} are satisfied for every finite subgroup of $\calN_{g,n}$ when $g\in\{4,5\}$ and $n\geqslant 2$, which gives the required inequality.

Let $F$ be any finite subgroup of $\modnn$. By \Cref{lem:COMPARE}, we have $\varphi(F)\cong F$, where $\varphi:\modnn\to\calN_g$ is the forgetful homomorphism. Then
we may regard $F$ as a finite
subgroup of $\calN_g$ and apply
\Cref{Prop:N4:HSSTN,Prop:N5:HSSTN} to the quotient orbifold $O_F=N_g/F$. 

Suppose first that $g=5$. By \Cref{Prop:N5:HSSTN,thm:OrbifoldMCG}, there exists
$k\in\{0,1\}$ such that
\[
\vcd\big(\Mod\big(\Sigma_{g_F,b_m+e_F}^{b_c}\big)\big)+\lambda(F)\leqslant \vcd(\calN_5)+k.
\]
If $\epsilon_F=0$, the minimum below is at most $2n/|F|\leqslant n$. If $\epsilon_F=1$, then $b_m=0$, and since $|F|\geqslant 2$ and $n\geqslant 2$, it follows that the sum of $\epsilon_F$ and that minimum equals $1+n/|F|\leqslant n$.
Since $k\leqslant1$, we obtain 
\[
    k+\epsilon_F+ \min\left\{\frac{2n}{|F|}, \frac{n}{|F|}+\frac{b_m}{2}\right\} \leqslant n+1.
\]
Therefore the hypotheses of \Cref{Bound:WF:by:OF} hold.

Now let $g=4$. By \Cref{Prop:N4:HSSTN,thm:OrbifoldMCG}, there exists
$k\in\{0,1,2\}$ such that
\[
\vcd\big(\Mod\big(\Sigma_{g_F,b_m+e_F}^{b_c}\big)\big)+\lambda(F)\leqslant \vcd(\calN_4)+k.
\]
If $k\leqslant 1$, the previous argument applies. If $k=2$, then we are in case~(4) of \Cref{Prop:N4:HSSTN}. 

The quotient signature is $(0;+;[-];\{(2,4,6)\})$, and hence $b_m=0$ and $\epsilon_F=0$. Since $|F|=48$ and $n\geqslant 2$, we obtain
\[
    k+\epsilon_F+ \min\left\{\frac{2n}{|F|}, \frac{n}{|F|}+\frac{b_m}{2}\right\} = 2+\frac{n}{48}\leqslant n+1.
\]
Therefore the hypotheses of \Cref{Bound:WF:by:OF} hold again.
\end{proof}

\section{Genus $\geqslant 6$}\label{sec:GENUS6}

In this section, we extend the results in \cite{HSSTN24} for $\calN_g$ to $\calN_{g,n}$ when $g\geqslant 6$. We follow the strategy of \cite[Section 6]{4Amigos_FullMCG}.

For convenience, in this section we work with $NF$ instead of $WF$. Since we are only interested in their virtual cohomological dimensions and $\vcd(NF)=\vcd(WF)$, they are interchangeable.

\begin{proposition}\cite[Theorems~5.1, 5.2, 5.4, $\&$ proof of Theorem~1.1]{HSSTN24}\label{prop:BOUNDCLOSED}
    For $g\geqslant 6$ and any finite subgroup $F$ of $\calN_g$ we have 
    \[
  \vcd\bigl(NF\bigr)+\lambda(F)
  \leqslant
  \vcd\bigl(\calN_g\bigr).
\]
\end{proposition}

\begin{theorem}\label{prop:Genus6}
For any $g\geqslant 6$, $n\geqslant 1$, and any finite subgroup $F$ of $\modnn$, we have 
\[
  \vcd(NF) + \lambda(F)
  \leqslant
  \vcd\bigl(\modnn\bigr).
\]
In particular,  
$\underline{\gd}(\modnn)=\cdfin(\modnn)= \vcd(\modnn)$.
\end{theorem}

\begin{proof}

Let $F$ be any finite subgroup of $\modnn$. By \Cref{lem:COMPARE}, $\varphi(F)$ is a finite subgroup of $\calN_g$ isomorphic to $F$, and $\lambda\bigl(\varphi(F)\bigr)=\lambda(F)$.
Moreover, notice that $\varphi(NF)\leqslant N\big(\varphi(F)\big)$. Restricting the Birman exact sequence~\eqref{eq:ses-full} to the normalizer $NF$ yields
\begin{equation*}\label{eq:ses-restricted}
1 \longrightarrow B_{n}(N_{g}) \cap NF 
  \longrightarrow NF 
  \xrightarrow{\varphi} 
  \varphi(NF)
  \longrightarrow 1.
\end{equation*}
Applying the subadditivity and monotonicity of $\vcd$ to the previous short exact sequence, we obtain
\[
  \vcd(NF)
  \leqslant
  \vcd\bigl(B_{n}(N_{g}) \cap NF\bigr)
  +
  \vcd\bigl(N\big(\varphi(F)\big)\bigr).
\]
Since $ \cdfin\big(B_{n}(N_{g})\big)=\text{cd}\big(B_{n}(N_{g})\big)=n+1$ by \cite[Theorem 1.2]{MR3869010}, we have
\begin{align*}
  \vcd(NF) + \lambda(F)
  &\leqslant (n+1) + \vcd\bigl(N\big(\varphi(F)\big)\bigr) + \lambda\bigl(\varphi(F)\bigr)\\[4pt]
  &\leqslant (n+1) + \vcd\bigl(\calN_g\bigr)= \vcd\bigl(\modnn\bigr),\end{align*}
where the last inequality is obtained by applying \Cref{prop:BOUNDCLOSED} to the finite subgroup $\varphi(F)$ when $g\geqslant 6$. The {\it in particular} part then follows from \Cref{Conchita:criterion} and \Cref{eq:InDim}.
\end{proof}

\appendix

\section{Cocompactness of Harer's spine}
\label{Appendix:cocompactness-spines}

Let $S_g$ be a connected closed orientable surface of genus \( g \), and take $\Delta=\{x_1,x_2,\ldots,x_s\}$ a set of $s\geqslant 1$ {\it marked points} in $S_g$. The {\it extended mapping class group} $\Gamma_{g,s}^{\pm}$ is the group of isotopy classes of all self-diffeomorphisms of $S_g$ which take the set of distinguished points $\Delta$ to itself. It is isomorphic to the extended mapping class group of the {\it punctured} surface $S_{g,s}:=S_g-\Delta$.  Moreover, the mapping class group $\Gamma_{g,s}$  (see \Cref{Sec:MCG}) is an index-two subgroup of $\Gamma_{g,s}^{\pm}$.

Let $s\geqslant 1$ and $2g+s>2$.  We recall next the description of a simplicial complex $\mathcal{Y} (=\mathcal{Y} _g^s)$ defined in \cite{Ha86}, and often called {\it Harer's spine} associated to $S_{g,s}$; see also \cite[Definition~3.2 $\&$ Section 3]{4Amigos_Harer}.

A $\Delta$-arc is a properly embedded path in $S_g$ with endpoints in $\Delta$.
An arc system is a finite set of isotopy classes relative to $X$ of
$\Delta$-arcs that admit representatives with pairwise disjoint interiors. We
require that, for every component $B$ obtained by splitting $S_g$ along
these representatives, the double of $B$ along $\partial B-\Delta$ has negative
Euler characteristic. An arc system
\emph{fills} $S_g$ if every point of $\Delta$ is an endpoint of an arc and every
component of its complement is homeomorphic to a disk. For more details see \cite{Ha86}[Section~1] and \cite[Section~2]{4Amigos_Harer}.

Harer's spine $\mathcal{Y}$ is the simplicial complex whose vertices are the filling
arc systems and whose $k$-simplices are the strict chains
  $A_0\subsetneq A_1\subsetneq\cdots\subsetneq A_k$
of filling arc systems.

The group $\Gamma_{g,s}^{\pm}$ acts simplicially on $\mathcal{Y}$, and  there is a $\Gamma_{g,s}^{\pm}$-equivariant deformation retraction of the {\it decorated Teichmüller space} $\mathcal{T}_{g,s}(\Delta)$ onto $\mathcal{Y}$; see for instance  \cite[Theorems~3.4 and~3.9]{4Amigos_Spine_NO}. Since the decorated
Teichm\"uller space is a model for $\underline{E}\Gamma_{g,s}^{\pm}$, so is
$\mathcal{Y}$; see \cite[Proposition~5.2]{4Amigos_Spine_NO}. We  now verify that this model is cocompact.

\begin{proposition}
\label{prop:orientable-spine-cocompact}
The complex $\mathcal{Y}$ is a cocompact model for
$\underline{E}\Gamma_{g,s}^{\pm}$.
\end{proposition}

\begin{proof} 
It suffices to show that the complex $\mathcal{Y}$ has finitely many
$\Gamma_{g,s}^{\pm}$-orbits of simplices.

Every arc system is contained in a maximal one, and every maximal arc system has $6g-6+3s$ arcs and cuts $S_g$ into $4g-4+2s$ triangles; see
\cite[Proposition~2.5, Lemma~2.6 and Proposition~2.7]{4Amigos_Harer}.
Moreover, there are only finitely many ways to pair the sides
of these triangles. On the other hand, maximal arc systems with the same gluing pattern
lie in the same $\Gamma_{g,s}^{\pm}$-orbit by the change of
coordinates principle. Therefore, there are finitely many orbits of
maximal arc systems; choose representatives
$ M_1,\ldots, M_r$.

Given a $k$-simplex $A_0\subsetneq A_1\subsetneq\cdots\subsetneq A_k$ of $\mathcal{Y}$, extend the largest arc system $A_k$ to a maximal arc system $M$. Some element of
$\Gamma_{g,s}^{\pm}$ takes this maximal arc system $M$ to one of the
$M_i$'s, and hence takes the given $k$-simplex to a strict
chain of filling arc subsystems of $ M_i$.
Each $M_i$ has finitely many subsets, so it contains
only finitely many strict chains of such subsets. Therefore, every
$k$-simplex orbit has a representative among the finitely many
%filling chains 
 strict chains of filling arc systems contained in $M_1,\ldots, M_r$.
Hence $\mathcal{Y}$ has finitely many
$\Gamma_{g,s}^{\pm}$-orbits of simplices and is therefore cocompact.
\end{proof}

We now use \Cref{prop:orientable-spine-cocompact}, to show the existence of a cocompact model for the classifying space for proper actions of the mapping class group of a non-orientable surface with punctures. 

Let $N_g$ be a
closed non-orientable surface of genus $g\geqslant 1$, and consider its orientable double cover
\[
  \pi\colon S_{g-1}\longrightarrow N_{g}, 
\]
with deck involution $\sigma:S_{g-1}\to S_{g-1}$. Let $n\geqslant 1$ and take $\Delta=\{x_1, \ldots, x_{2n}\}$ such that $\sigma(x_{2i}) = x_{2i-1}$ for all $1 \leqslant i \leqslant n$. Then $\sigma$ and $\pi$ restrict to $S_{g-1,2n}=S_{g-1}-\Delta$, and \( X:=\pi(\Delta) \)  is a collection of \( n\) distinguished points in \( N_g\).

 For $n\geqslant 1$ and $g+n>2$, let $\mathcal{Y}$ be Harer's spine associated with $S_{g-1,2n}$, and take the subcomplex $\mathcal{Z}:=\mathcal{Y}^\sigma$ of $\sigma$-fixed points. This is precisely the spine associated with $N_{g,n}:=N_g-X$  defined in
\cite[Corollary~3.11 and Definition~3.12]{4Amigos_Spine_NO}.

\begin{proposition}
\label{prop:nonorientable-spine-cocompact}
The spine $\mathcal{Z}$ associated with $N_{g,n}$ is a cocompact model for
$\underline{E}\calN_{g,n}$.
\end{proposition}

\begin{proof}
 Set $G=\Gamma_{g-1,2n}^{\pm}$ and 
let $H$ be the cyclic sugroup of order $2$ generated by the isotopy class of $\sigma$. By
 \Cref{prop:orientable-spine-cocompact}, the spine $\mathcal{Y}$ is a finite
$G$-CW-model for $\underline{E}G$. It follows from \cite[Lemma~1.3]{Luck00} that $\mathcal{Z}=\mathcal{Y}^H$ is a finite model for $\underline{E}(WH)$, where $WH=NH/H$ is the  Weyl group of $H$. 

 On the other hand, the orientable double cover $\pi\colon S_{g-1,2n}\longrightarrow N_{g,n}$ is characteristic and satisfies the Birman--Hilden property; see for instance \cite[Corollary 2]{5Amigos_BH}. Moreover, since $H$ is generated by the involution $\sigma$, its normalizer $NH$ is the subgroup of $\Gamma_{g-1,2n}^{\pm}$  consisting of mapping classes admitting representatives that preserve the fibers of $\pi$. Therefore the group $WH$ is isomorphic to $\calN_{g,n}$, and $\mathcal{Z}$ is a cocompact model for $\underline{E}\calN_{g,n}$.

\end{proof}

\begin{remark}
For every admissible choice of parameters, the spines $\mathcal{Y}_g^{s,m}$ of 
\cite[Definition 3.2]{4Amigos_Harer} and $\mathcal{Z}_{g+1}^{n,\ell}$ of
\cite[Definition 3.12]{4Amigos_Spine_NO} are cocompact under their corresponding mapping
class group actions. The proof is the same as for the cases $m=s$ and
$\ell=n$ above: one also records the punctured faces and, for
$\mathcal{Z}_{g+1}^{n,\ell}$, the action of $\sigma$ on the finite graph data.
\end{remark}

\section{Code used in the proof of \Cref{prop:upperbound:ProjectivePlane}}\label{Appendix:code:projective-plane}

\begin{lstlisting}[escapeinside={(*@}{@*)}]
from fractions import Fraction


def Omega(order):
    count = 0
    divisor = 2
    while divisor * divisor <= order:
        while order % divisor == 0:
            count += 1
            order //= divisor
        divisor += 1
    if order > 1:
        count += 1
    return count


(*@\color{red}@*)for n in list(range(3, 6)) + list(range(7, 18)):(*@\color{black}@*)
    order_bound = max(4 * n, 60)
    bounds = []
    for order in range(2, order_bound + 1):
        orbit_bound = min(
            Fraction(2 * n, order),
            Fraction(n, order) + Fraction(1, 2),
        )
        bounds.append((orbit_bound + Omega(order), order))

    maximum = max(bound for bound, _ in bounds)
    maximizing_orders = [
        order for bound, order in bounds if bound == maximum
    ]
(*@\color{red}@*)    if n <= 5:
        assert maximum < 6
        target = 6
    else:
        assert maximum < n - 1
        target = n - 1(*@\color{black}@*)
    print(
        f"n={n}: maximum bound = {maximum}, "
(*@\color{red}@*)        f"attained at |F|={maximizing_orders}; target < {target}"(*@\color{black}@*)
    )
\end{lstlisting}

\section{Code used in the proof of \Cref{prop:upperbound:KB}}\label{Appendix:code:n:F}

\begin{lstlisting}
from fractions import Fraction


def Omega(F):
    # We compute Omega(F) by trial division: every time a prime divisor
    # appears, it is counted once for each occurrence in the factorization.
    count = 0
    d = 2
    while d * d <= F:
        while F % d == 0:
            count += 1
            F //= d
        d += 1
    if F > 1:
        count += 1
    return count

for n in range(2, 14):
    # Check every integer order allowed by |F| <= 4n.
    values = [
        (min(Fraction(2*n, F), Fraction(n, F) + 1) + Omega(F), F)
        for F in range(2, 4*n + 1)
    ]
    maximum, order = max(values)
    print(f"n={n}: maximum bound = {maximum} at |F|={order}")

    if n in (2, 3, 4):
        expected = {2: Fraction(7, 2), 3: Fraction(15, 4), 4: Fraction(9, 2)}
        assert maximum == expected[n]
    assert maximum < (4 if n == 2 else n + 1)
\end{lstlisting}

\section{Code used in the proof of \Cref{prop:orbifolds-genus-3}}\label{Appendix:code:genus-3}

The following program produces the quotient orbifold candidates listed in
\Cref{table:orbifolds-genus-3}. For each finite group $F$ that can act on
$N_3$, it enumerates signatures using the notation introduced in
\cref{sec:genus-3} and keeps those satisfying the Riemann--Hurwitz equation (\ref{eq:RH}).

Since $|F|\geqslant 2$, one
has $\chi(\Sigma_F)\geqslant 0$, and hence $b\leqslant 2$. Each elliptic point
contributes at least $1/2$, so $e_F\leqslant 5$. If corner points occur, then
$|F|\geqslant 4$ and each contributes at least $1/4$, giving
$c_F\leqslant 5$. These are the bounds used in the program.

Thus the only possible underlying surfaces are the orientable surfaces of
genus $0$ or $1$ and the non-orientable surfaces of genus $1$ or $2$, with the
number of boundary components restricted by $\chi(\Sigma_F)\geqslant 0$. Moreover,
the order of an elliptic point is the order of a cyclic subgroup of $F$, while
a corner of order $q$ requires a dihedral subgroup $D_q\leqslant F$. The lists
in the program contain all such orders for each of the groups in
\cite[Proposition~1]{BEM14}. Finally, two boundary components cannot both
contain corners: in that case $b=2$, whereas two corners contribute at least
$1/2 \geq 1/|F|$. Hence the program considers every signature satisfying these
necessary conditions.
  
The realizability of these signatures 
requires a surface-kernel epimorphism from the corresponding NEC group onto
$F$, and is not claimed.

\begin{lstlisting}[
 language=Python,
 basicstyle=\ttfamily\footnotesize,
 breaklines=true,
 columns=fullflexible
]
from fractions import Fraction as F
from itertools import combinations_with_replacement


# G: (order, cone orders, corner orders, reflections allowed)
GROUPS = {
    "C2": (2, (2,), (), True),
    "C3": (3, (3,), (), False),
    "C4": (4, (2, 4), (), True),
    "C6": (6, (2, 3, 6), (), True),
    "D2": (4, (2,), (2,), True),  # D2 is isomorphic to C2 x C2
    "D3": (6, (2, 3), (3,), True),
    "D4": (8, (2, 4), (2, 4), True),
    "D6": (12, (2, 3, 6), (2, 3, 6), True),
}

# (orientability, genus, Euler characteristic of the closed surface)
SURFACES = (("+", 0, 2), ("+", 1, 0), ("-", 1, 1), ("-", 2, 0))


def choices(orders):
    for size in range(6):
        yield from combinations_with_replacement(orders, size)


def defect(orders, factor=1):
    return factor * sum((1 - F(1, n) for n in orders), F(0))


total = 0

for group, (order, cone_orders, corner_orders, reflections) in GROUPS.items():
    print(("\n" if total else "") + f"G = {group}")

    for sign, genus, chi_surface in SURFACES:
        for empty in range(3):
            for cones in choices(cone_orders):
                for corners in choices(corner_orders):
                    boundary = empty + bool(corners)
                    if boundary > 2 or (boundary and not reflections):
                        continue

                    chi = chi_surface - boundary
                    chi -= defect(cones) + defect(corners, F(1, 2))
                    if chi != -F(1, order):
                        continue

                    cone_data = ",".join(map(str, cones)) or "-"
                    cycles = [f"({','.join(map(str, corners))})"] if corners else []
                    boundary_data = ",".join(cycles + ["()"] * empty) or "-"

                    print(f"  ({genus}; {sign}; [{cone_data}]; {{{boundary_data}}})")
                    total += 1

print(f"\nTotal: {total}")
assert total == 36
\end{lstlisting}

\bibliographystyle{alpha} \bibliography{ref}

@article {GG12,
    AUTHOR = {Gon\c{c}alves, D. L. and Guaschi, J.},
     TITLE = {Surface braid groups and coverings},
   JOURNAL = {J. Lond. Math. Soc. (2)},
  FJOURNAL = {Journal of the London Mathematical Society. Second Series},
    VOLUME = {85},
      YEAR = {2012},
    NUMBER = {3},
     PAGES = {855--868},
      ISSN = {0024-6107,1469-7750},
   MRCLASS = {20F36 (57M07)},
  MRNUMBER = {2927811},
MRREVIEWER = {Tetsuya\ Ito},
       DOI = {10.1112/jlms/jdr071},
       URL = {https://doi-org.pbidi.unam.mx:2443/10.1112/jlms/jdr071},
}

@book {Lu89,
    AUTHOR = {L{\"u}ck, W.},
     TITLE = {Transformation groups and algebraic {$K$}-theory},
    SERIES = {Lecture Notes in Mathematics},
    VOLUME = {1408},
      NOTE = {Mathematica Gottingensis},
 PUBLISHER = {Springer-Verlag, Berlin},
      YEAR = {1989},
     PAGES = {xii+443},
      ISBN = {3-540-51846-0},
   MRCLASS = {57Sxx (18F25 19-02 57Q10 57Q12)},
  MRNUMBER = {1027600},
MRREVIEWER = {A. A. Ranicki},
       DOI = {10.1007/BFb0083681},
       URL = {https://doi.org/10.1007/BFb0083681},
}

@article {AMP14,
    AUTHOR = {Aramayona, {J}. and Mart\'{\i}nez-P\'{e}rez, {C}.},
     TITLE = {The proper geometric dimension of the mapping class group},
   JOURNAL = {Algebr. Geom. Topol.},
  FJOURNAL = {Algebraic \& Geometric Topology},
    VOLUME = {14},
      YEAR = {2014},
    NUMBER = {1},
     PAGES = {217--227},
      ISSN = {1472-2747},
   MRCLASS = {20F65 (20J05 57M07 57N05)},
  MRNUMBER = {3158758},
MRREVIEWER = {Sang-hyun Kim},
       DOI = {10.2140/agt.2014.14.217},
       URL = {https://doi.org/10.2140/agt.2014.14.217},
}

@article {MR3869010,
    AUTHOR = {Gon\c{c}alves, {D}. {L}. and Guaschi, {J}. and Maldonado,
              {M}.},
     TITLE = {Embeddings and the (virtual) cohomological dimension of the
              braid and mapping class groups of surfaces},
   JOURNAL = {Confluentes Math.},
  FJOURNAL = {Confluentes Mathematici},
    VOLUME = {10},
      YEAR = {2018},
    NUMBER = {1},
     PAGES = {41--61},
   MRCLASS = {57N05 (20F36 20F38 20J06 55P20 55R80 57M07)},
  MRNUMBER = {3869010},
MRREVIEWER = {B\l a\.{z}ej Szepietowski},
       DOI = {10.5802/cml.45},
       URL = {https://doi-org.pbidi.unam.mx:2443/10.5802/cml.45},
}

@article {Ha86,
    AUTHOR = {Harer, J. L.},
     TITLE = {The virtual cohomological dimension of the mapping class group
              of an orientable surface},
   JOURNAL = {Invent. Math.},
  FJOURNAL = {Inventiones Mathematicae},
    VOLUME = {84},
      YEAR = {1986},
    NUMBER = {1},
     PAGES = {157--176},
      ISSN = {0020-9910},
   MRCLASS = {32G15 (20F38 57N05)},
  MRNUMBER = {830043},
MRREVIEWER = {K. Vogtmann},
       DOI = {10.1007/BF01388737},
       URL = {https://doi-org.pbidi.unam.mx:2443/10.1007/BF01388737},
}

@article{VcdBraid,
  author  = {Gon{\c{c}}alves, {D}. {L}. and Guaschi, {J}.},
  title   = {Inclusion of configuration spaces in Cartesian products,
             and the virtual cohomological dimension of the braid groups
             of {$S^2$} and {${\mathbb R}P^2$}},
  journal = {Pacific Journal of Mathematics},
  volume  = {287},
  number  = {1},
  year    = {2017},
  pages   = {71--99},
  doi     = {10.2140/pjm.2017.287.71}
}

@incollection {BraidSurvey,
    AUTHOR = {Guaschi, J. and Juan-Pineda, D.},
     TITLE = {A survey of surface braid groups and the lower algebraic
              {$K$}-theory of their group rings},
 BOOKTITLE = {Handbook of group actions. {V}ol. {II}},
    SERIES = {Adv. Lect. Math. (ALM)},
    VOLUME = {32},
     PAGES = {23--75},
 PUBLISHER = {Int. Press, Somerville, MA},
      YEAR = {2015},
   MRCLASS = {20F36 (19A31 19B28)},
  MRNUMBER = {3382024},
MRREVIEWER = {Inasa Nakamura},
}

@InCollection{Lu05,
  Title                    = {Survey on classifying spaces for families of subgroups},
  Author                   = {L{\"u}ck, W.},
  Booktitle                = {Infinite groups: geometric, combinatorial and dynamical aspects},
  Publisher                = {Birkh\"auser, Basel},
  Year                     = {2005},
  Pages                    = {269--322},
  Series                   = {Progr. Math.},
  Volume                   = {248}
}

@article {BEM14,
    AUTHOR = {Bujalance, E. and Etayo, J. J. and Mart\'inez, E.},
     TITLE = {The full group of automorphisms of non-orientable unbordered
              {K}lein surfaces of topological genus 3, 4 and 5},
   JOURNAL = {Rev. Mat. Complut.},
  FJOURNAL = {Revista Matem\'atica Complutense},
    VOLUME = {27},
      YEAR = {2014},
    NUMBER = {1},
     PAGES = {305--326},
      ISSN = {1139-1138,1988-2807},
   MRCLASS = {57M60 (20B25)},
  MRNUMBER = {3149189},
MRREVIEWER = {Adnan\ Meleko\u glu},
       DOI = {10.1007/s13163-013-0121-7},
       URL = {https://doi-org.pbidi.unam.mx:2443/10.1007/s13163-013-0121-7},
}

@article {CX24,
    AUTHOR = {Colin, N. and Xicot\'encatl, M. A.},
     TITLE = {The {N}ielsen realization problem for non-orientable surfaces},
   JOURNAL = {Topology Appl.},
  FJOURNAL = {Topology and its Applications},
    VOLUME = {353},
      YEAR = {2024},
     PAGES = {Paper No. 108957, 13},
      ISSN = {0166-8641,1879-3207},
   MRCLASS = {57K20 (30F10 30F50 30F60 32G15 57M07)},
  MRNUMBER = {4755494},
       DOI = {10.1016/j.topol.2024.108957},
       URL = {https://doi.org/10.1016/j.topol.2024.108957},
}

@article {Conway,
    AUTHOR = {Conway, {J}.{H}. and Huson, {D}.{H}.},
     TITLE = {The Orbifold Notation for Two-Dimensional Groups},
   JOURNAL = {Structural Chemistry},
  FJOURNAL = {M\"{u}nster Journal of Mathematics},
    VOLUME = {13},
      YEAR = {2002},
     PAGES = {247--257},
}

@article {5Amigos_BH,
    AUTHOR = {Colin, {N}. and Hidalgo, R. and Jim\'enez Rolland, R. and Morales, I. and Quispe, S.},
     TITLE = {Birman--{H}ilden theory for big mapping class groups},
   JOURNAL = {ArXiv preprint 2601.09897},
      YEAR = {2026},
}

@article {4Amigos_FullMCG,
    AUTHOR = {Colin, {N}. and Jim\'enez Rolland, R. and Le\'on
              \'Alvarez, P. L. and S\'anchez Salda\~na, L. J.},
     TITLE = {The proper geometric dimension of the mapping class group of
              an orientable surface with punctures},
   JOURNAL = {J. Group Theory},
  FJOURNAL = {Journal of Group Theory},
    VOLUME = {28},
      YEAR = {2025},
    NUMBER = {6},
     PAGES = {1477--1500},
      ISSN = {1433-5883,1435-4446},
   MRCLASS = {20F65 (20F38 20J06)},
  MRNUMBER = {4979231},
       DOI = {10.1515/jgth-2024-0227},
       URL = {https://doi.org/10.1515/jgth-2024-0227},
}

@incollection {Earle,
    AUTHOR = {Earle, C. J.},
     TITLE = {Diffeomorphisms and automorphisms of compact hyperbolic
              2-orbifolds},
 BOOKTITLE = {Geometry of {R}iemann surfaces},
    SERIES = {London Math. Soc. Lecture Note Ser.},
    VOLUME = {368},
     PAGES = {139--155},
 PUBLISHER = {Cambridge Univ. Press, Cambridge},
      YEAR = {2010},
   }

@article {HSSTN24,
    AUTHOR = {Hidber, C. E. and S\'anchez Salda\~na, L. J. and
              Trujillo-Negrete, A.},
     TITLE = {On the dimensions of mapping class groups of non-orientable
              surfaces},
   JOURNAL = {Homology Homotopy Appl.},
  FJOURNAL = {Homology, Homotopy and Applications},
    VOLUME = {24},
      YEAR = {2022},
    NUMBER = {1},
     PAGES = {347--372},
      ISSN = {1532-0073,1532-0081},
   MRCLASS = {20F34 (20F65 20J05)},
  MRNUMBER = {4432529},
MRREVIEWER = {Mustafa\ Korkmaz},
       DOI = {10.4310/hha.2022.v24.n1.a17},
       URL = {https://doi-org.pbidi.unam.mx:2443/10.4310/hha.2022.v24.n1.a17},
}

@article {Zieschang73,
    AUTHOR = {Zieschang, H.},
     TITLE = {On the homeotopy group of surfaces},
   JOURNAL = {Math. Ann.},
  FJOURNAL = {Mathematische Annalen},
    VOLUME = {206},
      YEAR = {1973},
     PAGES = {1--21},
      ISSN = {0025-5831,1432-1807},
   MRCLASS = {30A60 (57A05)},
  MRNUMBER = {335794},
MRREVIEWER = {J.\ S.\ Birman},
       DOI = {10.1007/BF01431525},
       URL = {https://doi.org/10.1007/BF01431525},
}

@article {BLN01,
    AUTHOR = {Brady, {N}. and Leary, {I}. J. and Nucinkis, {B}. {E}. {A}.},
     TITLE = {On algebraic and geometric dimensions for groups with torsion},
   JOURNAL = {J. London Math. Soc. (2)},
  FJOURNAL = {Journal of the London Mathematical Society. Second Series},
    VOLUME = {64},
      YEAR = {2001},
    NUMBER = {2},
     PAGES = {489--500},
      ISSN = {0024-6107},
   MRCLASS = {20J05 (20F05 20F65 57M07)},
  MRNUMBER = {1853466},
       DOI = {10.1017/S002461070100246X},
}

@article {4Amigos_Spine_NO,
    AUTHOR = {Colin, {N}. and Jim\'enez Rolland, {R}. and Le\'on
              \'Alvarez, {P}. {L}. and S\'anchez Salda\~na, {L}. {J}.},
     TITLE = {A spine for the decorated {T}eichm\"uller space of a punctured
              non-orientable surface},
   JOURNAL = {Bull. Lond. Math. Soc.},
  FJOURNAL = {Bulletin of the London Mathematical Society},
    VOLUME = {57},
      YEAR = {2025},
    NUMBER = {12},
     PAGES = {3749--3767},
      ISSN = {0024-6093,1469-2120},
   MRCLASS = {57K20 (20J05 55R35 57M07 57M60)},
  MRNUMBER = {5007373},
       DOI = {10.1112/blms.70228},
       URL = {https://doi-org.pbidi.unam.mx:2443/10.1112/blms.70228},
}

@article {4Amigos_Harer,
    AUTHOR = {Colin, {N}. and Jim\'enez Rolland, {R}. and Le\'on
              \'Alvarez, {P}. {L}. and S\'anchez Salda\~na, {L}. {J}.},
     TITLE = {On the dimension of {H}arer's spine for the decorated
              {T}eichm\"uller space},
   JOURNAL = {Homology Homotopy Appl.},
  FJOURNAL = {Homology, Homotopy and Applications},
    VOLUME = {28},
      YEAR = {2026},
    NUMBER = {2},
     PAGES = {67--78},
}

@article {Luck00,
    AUTHOR = {L{\"u}ck, {W}.},
     TITLE = {The type of the classifying space for a family of subgroups},
   JOURNAL = {J. Pure Appl. Algebra},
  FJOURNAL = {Journal of Pure and Applied Algebra},
    VOLUME = {149},
      YEAR = {2000},
     PAGES = {177--203},
}

@article {Yagasaki,
    AUTHOR = {Yagasaki, T.},
     TITLE = {Homotopy types of homeomorphism groups of noncompact
              {$2$}-manifolds},
   JOURNAL = {Topology Appl.},
  FJOURNAL = {Topology and its Applications},
    VOLUME = {108},
      YEAR = {2000},
    NUMBER = {2},
     PAGES = {123--136},
      ISSN = {0166-8641,1879-3207},
   MRCLASS = {57N05 (57N20 57S05)},
  MRNUMBER = {1787857},
MRREVIEWER = {V.\ P.\ Snaith},
       DOI = {10.1016/S0166-8641(99)00130-3},
       URL = {https://doi.org/10.1016/S0166-8641(99)00130-3},
}

@article {Gramain,
    AUTHOR = {Gramain, A.},
     TITLE = {Le type d'homotopie du groupe des diff\'eomorphismes d'une
              surface compacte},
   JOURNAL = {Ann. Sci. \'Ecole Norm. Sup. (4)},
  FJOURNAL = {Annales Scientifiques de l'\'Ecole Normale Sup\'erieure.
              Quatri\`eme S\'erie},
    VOLUME = {6},
      YEAR = {1973},
     PAGES = {53--66},
      ISSN = {0012-9593},
   MRCLASS = {58D05 (32G15)},
  MRNUMBER = {326773},
MRREVIEWER = {J.\ E.\ Marsden},
       URL = {http://www.numdam.org/item?id=ASENS_1973_4_6_1_53_0},
}

@article {Stu06,
    AUTHOR = {Stukow, {M}.},
     TITLE = {Dehn twists on nonorientable surfaces},
   JOURNAL = {Fund. Math.},
  FJOURNAL = {Fundamenta Mathematicae},
    VOLUME = {189},
      YEAR = {2006},
    NUMBER = {2},
     PAGES = {117--147},
      ISSN = {0016-2736,1730-6329},
   MRCLASS = {57N05 (57M50)},
  MRNUMBER = {2214574},
MRREVIEWER = {Bruno\ P.\ Zimmermann},
       DOI = {10.4064/fm189-2-3},
       URL = {https://doi-org.pbidi.unam.mx:2443/10.4064/fm189-2-3},
}

@article {stukow,
    AUTHOR = {Stukow, {M}.},
     TITLE = {Conjugacy classes of finite subgroups of certain mapping class
              groups},
   JOURNAL = {Turkish J. Math.},
  FJOURNAL = {Turkish Journal of Mathematics},
    VOLUME = {28},
      YEAR = {2004},
    NUMBER = {2},
     PAGES = {101--110},
      ISSN = {1300-0098,1303-6149},
   MRCLASS = {57M60 (20E45 20F38 57N05)},
  MRNUMBER = {2062555},
MRREVIEWER = {Mustafa\ Korkmaz},
}

@article {MAHER,
    AUTHOR = {Maher, {J}.},
     TITLE = {Random walks on the mapping class group},
   JOURNAL = {Duke Math. J.},
  FJOURNAL = {Duke Mathematical Journal},
    VOLUME = {156},
      YEAR = {2011},
    NUMBER = {3},
     PAGES = {429--468},
      ISSN = {0012-7094,1547-7398},
   MRCLASS = {37E30 (20F65 20H10 32G15 37A50 60G50)},
  MRNUMBER = {2772067},
MRREVIEWER = {Jayadev\ S.\ Athreya},
       DOI = {10.1215/00127094-2010-216},
       URL = {https://doi.org/10.1215/00127094-2010-216},
}

@article {KOR,
    AUTHOR = {Korkmaz, {M}.},
     TITLE = {Mapping class groups of nonorientable surfaces},
   JOURNAL = {Geom. Dedicata},
  FJOURNAL = {Geometriae Dedicata},
    VOLUME = {89},
      YEAR = {2002},
     PAGES = {109--133},
      ISSN = {0046-5755,1572-9168},
   MRCLASS = {57N05 (20F38 57M07)},
  MRNUMBER = {1890954},
MRREVIEWER = {Stephen\ P.\ Humphries},
       DOI = {10.1023/A:1014289127999},
       URL = {https://doi.org/10.1023/A:1014289127999},
}

\end{document}